\documentclass[11pt]{amsart}
\usepackage{amsfonts,amssymb,amscd,amsmath,enumerate,verbatim,calc}

\usepackage{amsmath}
\usepackage{amssymb}
\usepackage{amscd}

\usepackage{color}

\def\NZQ{\mathbb}               
\def\NN{{\NZQ N}}
\def\QQ{{\NZQ Q}}
\def\ZZ{{\NZQ Z}}
\def\RR{{\NZQ R}}
\def\CC{{\NZQ C}}

\def\PP{{\NZQ P}}

\newtheorem{Theorem}{Theorem}[section]
\newtheorem{Lemma}[Theorem]{Lemma}
\newtheorem{Corollary}[Theorem]{Corollary}
\newtheorem{Proposition}[Theorem]{Proposition}
\newtheorem{Remark}[Theorem]{Remark}

\newtheorem{Example}[Theorem]{Example}

\newtheorem{Definition}[Theorem]{Definition}

\let\epsilon\varepsilon
\let\phi=\varphi
\let\kappa=\varkappa

\def \mm{\mathfrak m}

\def\mR{m_R}
\def\Ass{\operatorname{Ass}}
\def\mfp{\mathfrak p}
\def\mfa{\mathfrak a}
\def\Spec{\operatorname{Spec}}
\def \Min{\operatorname{Min}}
\def\liminf{\operatorname {liminf}}
\def\depth{\operatorname{depth}}
\def\height{\operatorname{height}}
\begin{document}
\title{Analytic Spread of  Filtrations and Symbolic Algebras}

\author{Steven Dale Cutkosky}
\author{Parangama Sarkar}

\thanks{The first author was partially supported by NSF grant DMS-2054394.}
\thanks{The second author was partially supported by the DST, India: INSPIRE Faculty Fellowship.}

\address{Steven Dale Cutkosky, Department of Mathematics,
University of Missouri, Columbia, MO 65211, USA}
\email{cutkoskys@missouri.edu}

\address{Parangama Sarkar,  Department of Mathematics,
Indian Institute of Technology, Palakkad, India}
\email{parangama@iitpkd.ac.in}

 \begin{abstract} In this paper we define and explore the analytic spread $\ell(\mathcal I)$ of a filtration in a local ring. We show that, especially for divisorial and symbolic filtrations, some basic properties of the analytic spread of an ideal extend to filtrations, even when the filtration is non Noetherian. We also illustrate some significant differences between the analytic spread of a filtration and the analytic spread of an ideal with examples.
 
 In the case of an ideal $I$, we have the classical bounds $\mbox{ht}(I)\le\ell(I)\le \dim R$. The upper bound $\ell(\mathcal I)\le \dim R$ is true for filtrations $\mathcal I$, but the lower bound is not true for all filtrations. We show that for the filtration $\mathcal I$ of symbolic powers of a height two prime ideal $\mathfrak p$ in a regular local ring of dimension three (a space curve singularity), so that $\mbox{ht}(\mathcal I) =2$ and $\dim R=3$, we have that $0\le \ell(\mathcal I)\le 2$ and all values of 0,1 and 2 can occur. In the cases of analytic spread 0 and 1 the symbolic algebra is necessarily non-Noetherian. The symbolic algebra is  non-Noetherian if and only if $\ell(\mathfrak p^{(n)})=3$ for all symbolic powers of  $\mathfrak p$ and if and only if $\ell(\mathcal I_a)=3$ for all 
 truncations $\mathcal I_a$ of $\mathcal I$.
\end{abstract}

\subjclass[2010]{13A02, 13A15, 13A18}

\maketitle



\section{Introduction}
  
 The analytic spread of an ideal $I$ in a (Noetherian) local ring $R$ is defined to be
  \begin{equation}\label{eqN0}
  \ell(I)=\dim R[I]/m_RR[I]
  \end{equation}
  where $R[I]=\bigoplus_{n\ge 0}I^n$ is the Rees algebra of $I$.  
  
  We recall some basic properties of analytic spread from \cite{Li} and \cite{HS}.
  We  have that upper semicontinuity of fiber dimension holds, that is 
\begin{equation}\label{eqN2}
\ell(I_P)\le \ell(I_{P'})\mbox{ if }P\subset P'\mbox{ are prime ideals containing $I$}.
\end{equation}
This follows for instance by \cite[(IV.13.1.5)]{EGA}.

  We have  inequalities (\cite[page 115]{Li} and \cite[Corollary 8.3.9]{HS})
  \begin{equation}\label{eqN1}
  \ell(I)\le \dim R
  \end{equation}
  and
  \begin{equation}\label{eqN3}
   {\rm ht}(I)\le \ell(I).
   \end{equation}
  The lower bound (\ref{eqN3}) follows from (\ref{eqN2}) since at a minimal prime $Q$ of $I$, we have that 
   $\ell(I_Q)=\mbox{ht}(Q)\ge \mbox{ht}(I)$ since $I_Q$ is $Q_Q$-primary.
   
     An ideal $I$ in a local ring $R$ for which the equality $\mbox{ht}(I)=\ell(I)$ holds is called equimultiple. $I$ is equimultiple if and and only if all fibers of $\pi_0:\mbox{Proj}(R[I]/IR[I])\rightarrow \mbox{Spec}(R/I)$ have the same dimension. This follows since 
     if $I$ is equimultiple and   $P$ is a prime ideal of $R$ which contains $I$, then by (\ref{eqN3}) and (\ref{eqN2}),
 $$
  \mbox{ht}(I)\le\mbox{ht}(I_P) \le \ell(I_P)
 \le\ell(I)
 =\mbox{ht}(I).
  $$
 In particular, if $I$ is an equimultiple ideal, then
 \begin{equation}\label{eqN7*}
 \ell(I_P)=\mbox{ht}(I_P)
 \end{equation}
 for all  prime ideals $P$ containing $I$.   For the other direction, we consider a minimal prime $P$ of $I$ such that $\mbox{ht}(I)=\mbox{ht}(P)$. If all fibers of $\pi_0$ have the same dimension then we have $\mbox{ht}(I)=\mbox{ht}(I_P)=\ell(I_P)=\ell(I)$.

We have the following fundamental theorem.

\begin{Theorem}\label{ThmN10}(\cite{Mc}, \cite[Theorem 5.4.6]{HS}) Let $R$ be a formally equidimensional  local ring and $I$ be an ideal in $R$. Then $m_R\in \mbox{Ass}(R/\overline{I^n})$ for some $n$ if and only if $\ell(I)=\dim R$.
\end{Theorem}






In this paper we extend the analytic spread of an ideal in a local ring to (not necessarily Noetherian) filtrations, and explore generalizations of the above results to general filtrations, divisorial filtrations and filtrations of symbolic powers. 


Let $\mathcal I=\{I_n\}$ be a filtration on a local ring $R$. The Rees ring of the filtration is $R[\mathcal I]=\oplus_{n\ge 0}I_n$. Analogously to the case of ideals, we define the analytic spread of the filtration to be
\begin{equation}\label{eqN10}
\ell(\mathcal I)=\dim R[\mathcal I]/m_RR[\mathcal I].
\end{equation}
We show in Lemma \ref{LemmaN4}, that the upper bound  (\ref{eqN1}) holds for filtrations $\mathcal I$, that is, 
$$
\ell(\mathcal I)\le \dim R.
$$
For an arbitrary filtration, we have that $\sqrt{I_n}=\sqrt{I_1}$ for all $n$ (equation (\ref{dim*})) and we  define the height of a filtration $\mathcal I$ to be 
$$
\mbox{ht}(\mathcal I)=\mbox{ht}(I_1).
$$
We may call a filtration $\mathcal I$  equimultiple if $\mbox{ht}(\mathcal I)=\ell(\mathcal I)$.

 A simple example of a filtration for which the lower bound (\ref{eqN3}) is not true is the following.
\begin{Example}\label{ExampleN6} Let $R$ be a local ring of dimension greater than zero. Let $\mathcal I=\{I_n\}$ where $I_n=m_R$ for $n\ge 1$. Then  $\ell(\mathcal I)=0<\mbox{ht}(\mathcal I)=\dim R$.
\end{Example}

In Example \ref{ExampleN6}, all ideals $I_n$ and all truncations (Noetherian approximations) $\mathcal I_a$ of $\mathcal I$ are equimultiple even though $\mathcal I$ is not. This example shows that the ``only if'' direction of Theorem \ref{ThmN10} can fail for filtrations.

In the case that $\mathcal I$ is a Noetherian filtration,  the lower bound ${\rm ht}(\mathcal I)\le \ell(\mathcal I)$ always holds
 (Proposition \ref{PropAS11}), so that the  inequality (\ref{eqN3}) for ideals continues to hold for Noetherian filtrations. 
 
 The condition that  a filtration has  analytic spread zero has a simple ideal theoretic interpretation (Lemma \ref{sp0}).
 Suppose that  $\mathcal I=\{I_n\}$ is a filtration in a local ring $R$. Then the analytic spread
$\ell(\mathcal I)=0$ if and only if
$$
\mbox{For all $n>0$ and $f\in I_n$, there exists $m>0$ such that $f^m\in m_RI_{mn}$.}
$$




  We consider (integral) divisorial filtrations and $s$-divisorial filtrations  in Section
   \ref{SecDiv}.  Divisorial and $s$-divisorial filtrations are defined at the beginning of this section.
One of the fundamental properties about an $m_R$-primary ideal $I$ is that $\ell(I)=\dim R$. We saw in Example \ref{ExampleN6} that this property fails for general filtrations. However, it is true for divisorial filtrations of $m_R$-primary ideals ($0$-divisorial filtrations). The following theorem  shows that the ``only if'' direction of Theorem \ref{ThmN10} holds for divisorial filtrations of $m_R$-primary ideals. 
  
 \begin{Theorem}\label{ThmN21*}(Theorem \ref{ThmN21})  Suppose that $R$ is a $d$-dimensional excellent local domain and $\mathcal I$ is a divisorial filtration of $m_R$-primary ideals on $R$. Then $\ell(\mathcal I)=d$.
   \end{Theorem} 
   
   Further, the  ``if" statement of  Theorem \ref{ThmN10} is true for divisorial filtrations.
   
   \begin{Theorem}\label{PropAS7*}(Theorem \ref{PropAS7}) Suppose that $R$ is a  local domain and $\mathcal I=\{I_n\}$ is a divisorial filtration on $R$. Let $I_n=I(\nu_1)_{a_1n}\cap\cdots\cap I(\nu_r)_{a_rn}$ for $n\ge 1$, some valuations $\nu_i$ and some $a_1,\ldots,a_r\in \ZZ_{>0}$. Suppose that $\ell(\mathcal I)=\dim R$. Then for some $\nu_i$, the center $m_{\nu_i}\cap R=\{f\in R\mid \nu_i(f)>0\}$   is $m_R$.	There exists a positive integer $n_0$ such that $m_R$ is an associated prime of $I_n=\overline{I_n}$ for all $n\geq n_0$.
	\end{Theorem}

Suppose that $\mathcal I=\{I_n\}$ is a filtration in $R$ and $\mathfrak p$ is a prime ideal in $R$. Then the localization of $\mathcal I$ at $\mathfrak p$ is the filtration $\mathcal I_{\mathfrak p}=\{(I_n)_{\mathfrak p}\}$ in $R_{\mathfrak p}$.   In a filtration $\mathcal I=\{I_n\}$, the ideals $I_n$ have the same minimal primes for all $n\ge 1$.
   
\begin{Corollary}\label{CorN31*}(Corollary \ref{CorN31}) Suppose that $R$ is a  local domain and $\mathcal I=\{I_n\}$ is an $s$-divisorial filtration on $R$ (a divisorial filtration consisting of ideals which are equidimensional of dimension $s$ and have no embedded components). Then $\ell(\mathcal I_Q)<\dim(R_Q)$ for all prime ideals $Q$ of $R$ which are not minimal primes of $I_1$.
\end{Corollary}


The $a$-the truncation of a filtration $\mathcal I$ is the Noetherian approximation of $\mathcal I$ generated by the first $a$ terms of $\mathcal I$. A formal definition of a truncation is given in Definition \ref{trunc}.
Knowledge of the analytic spread of the truncations of a filtration can give some information about the analytic spread of the filtration, as is illustrated in the following corollary to Theorem \ref{PropAS7}.

\begin{Corollary}\label{SmSpread*}(Corollary \ref{SmSpread})
	Let $R$ be a local domain and $\mathcal I=\{I_n\}$ be a divisorial filtration in $R$ where $I_n=\bigcap_{i=1}^rI(\nu_i)_{na_i}$ for all $n\geq 1.$ Suppose $I_1=\cap_{i=1}^rI(\nu_i)_{a_i}$ is a minimal primary decomposition of $I_1$ and $\ell(\mathcal I_a)<\dim R$ for some $a\geq 1$ where $\mathcal I_a$ is the $a$-th truncated filtration of $\mathcal I.$ Then $\ell(\mathcal I)<\dim R.$
\end{Corollary}


We turn to  symbolic algebras in Section \ref{SecSym}.	Let $I$ be an ideal in a local ring $R$. For $n$ a positive integer, the $n$-th symbolic power $I^{(n)}$ of $I$ is
 $$
 I^{(n)}=\cap_{\mfp \in \mbox{Ass}(R/I)}(I^nR_{\mfp}\cap R).
 $$
 
 Symbolic algebras and filtrations have been extensively studied. A survey of some recent progress is given in \cite{DDGHN}.

We consider in Section \ref{SecSym}
	 the filtration of symbolic powers $\{I^{(n)}\}$ where $I=P_1\cap\cdots \cap P_r$ is an intersection of prime ideals of a common height in an excellent  local ring. If $P$ is a prime ideal in a regular local ring $R$, then since $R_{P}$ is a regular local ring, the $P_{P}$-adic order on $R_{P}$ defines a discrete valuation $\nu$ of the quotient field of $R$  such that the valuation ideals $I(\nu)_n$ of $R$ are the symbolic powers $I(\nu)_n=P^{(n)}$. Thus the symbolic filtrations in a regular local ring are  divisorial filtrations. 
	 
There are examples of height two prime ideals $P$ in an excellent regular local ring $R$ of dimension three (space curve singularities) such that the symbolic algebra of $P$, $\oplus_{n\ge 0}P^{(n)}$, is not a finitely generated $R$-algebra \cite{Ro}, and even when $P$ is analytically irreducible \cite{GNW} and \cite{KN}.
	 	 
We have  a simple characterization of when a symbolic filtration is Noetherian in terms of analytic spread.
  Suppose that $I\subset J$ are proper ideals in a local ring $R$. Define
 $S_J(I)=\oplus_{k\ge 0}I^k:J^{\infty}$ where $I^k:J^{\infty}=\cup_{i=1}^{\infty}I^k:_RJ^i$.
 Let $J$ be the intersection of all asymptotic prime divisors of $I$ which are not minimal primes. Then
 $I^{(n)}=I^n:J^{\infty}$ and the symbolic algebra $\oplus_{n\ge 0}I^{(n)}=S_J(I)$.
 In the case that $I=P$ is the ideal of a space curve singularity, the symbolic algebra is $\oplus_{n\ge 0}P^{(n)}=S_{m_R}(P)$.


\begin{Theorem}(\cite[Theorem 2.6]{CHS}, Theorem \ref{strengthen})
\label{strengthen*}
Let $(R,\mm)$ be an excellent domain, and let $I$ and $J$ be proper ideals of $R$.  Then the following conditions are equivalent:
\begin{enumerate}
\item[{\em (a)}]  $S_J(I)$ is finitely generated.
\item[{\em (b)}]  There exists an integer $r>0$ such that $\ell((I^r: J^\infty)_P)<\dim R_P$ for all $P\in V(J)$.
\end{enumerate}
\end{Theorem}

A  related result  for ordinary symbolic powers was proven by Katz and Ratliff in Theorem A and Corollary~1 of \cite{KR}.

We have the following immediate corollary.


   \begin{Corollary}\label{LemmaN6*}(Corollary \ref{LemmaN6}) Suppose that $R$ is an excellent  local domain of dimension $d$ and $I=P_1\cap \cdots \cap P_r$ is an intersection of
    prime ideals $P_i$ of $R$ of  a common height. Then the ring $\bigoplus_{n\ge 0}I^{(n)}$ is a finitely generated $R$-algebra if and only if there exists $n\in \ZZ_{>0}$ such that the analytic spread $\ell( I^{(n)}_Q)<\mbox{ht}(Q)$ for all prime ideals $Q$ of $R$ which contain $I$ and are not one of the minimal primes $P_i$ of $I$.  
   \end{Corollary}
   
   
 With our assumption that $R$ is an excellent  local ring of dimension $d$ and $I=P_1\cap \cdots \cap P_r$ is an intersection of prime ideals $P_i$ of $R$ of  a common height in Corollary \ref{LemmaN6*}, we have that  $\ell( I^{(n)}_{P_i})=\mbox{ht}(P_i)=\mbox{ht}(I)$  for the minimal primes $P_i$ of $I$.

 \begin{Corollary}\label{CorN70*}(Corollary \ref{CorN70}) Suppose that $R$ is an excellent  local domain of dimension $d$ and $I=P_1\cap \cdots \cap P_r$ is an intersection of
    prime ideals $P_i$ of $R$ of  a common height. If $I^{(n)}$ is equimultiple for some $n$ then the symbolic algebra $\oplus_{n\ge 0}I^{(n)}$ is a finitely generated $R$-algebra.
    \end{Corollary}

    
    However, there exist ideals $I$ such that the symbolic algebra $\oplus_{n\ge 0}I^{(n)}$ is a finitely generated $R$-algebra but no symbolic power $I^{(n)}$ is equimultiple (Example \ref{ExAS1}).

  In contrast to the conclusions of Corollary \ref{LemmaN6}, we have that  inequality of analytic spread and height   $\ell(\mathcal I_Q)<\mbox{ht}(Q)$ holds at all non minimal primes for  symbolic filtrations, irregardless of whether their symbolic algebra is a finitely generated $R$-algebra. 
 The following proposition follows from Corollary \ref{CorN31}.

   \begin{Proposition}\label{PropN22*}(Proposition \ref{PropN22}) Suppose that $R$ is a local domain of dimension $d$ and $I=P_1\cap \cdots \cap P_r$ is an intersection of
    prime ideals $P_i$ of $R$ of  a common height. Suppose $R_{P_i}$ is a regular local ring for $1\leq i\leq r$. Let $\mathcal I=\{I^{(n)}\}$ be the symbolic filtration of $I$. Then the analytic spread $\ell( \mathcal I_Q)<\mbox{ht}(Q)$ for all prime ideals $Q$ of $R$ which contain $I$ and are not one of the minimal primes $P_i$ of $I$ and $\ell(\mathcal I_{P_i})=\mbox{ht}(P_i)=\mbox{ht}(I)$
    for all minimal primes  $P_i$ of $I$.
    \end{Proposition}


The following  theorem shows that 
 in the case of the symbolic algebra of a height two prime ideal in a three dimensional local ring, the analytic spread of the symbolic filtration is bounded above by the height, which is 2, and all analytic spreads $\le 2$ occur. Thus the inequality of (\ref{eqN3}) for ideals is reversed! In contrast, even for non Noetherian divisorial filtrations of $m_R$-valuations in an excellent local domain, the analytic spread of the filtration must be equal to dimension $R$ by Theorem \ref{ThmN21*}.

 \begin{Theorem}\label{PropN5*}(Theorem \ref{PropN51})  Suppose that $R$ is a regular local ring of dimension 3, and $\mathfrak p$ is a height two prime ideal of $R$. Let $\mathcal I=\{\mathfrak p^{(n)}\}$ be the symbolic filtration. Then $\ell(\mathcal I)\le 2$ and all values $\ell(\mathcal I)=0,1,2$ can occur.
  \end{Theorem}
  In Section \ref{SecEx} we construct examples illustrating this theorem with $\ell(\mathcal I)=0$ and 1. A simple example with $\ell(\mathcal I)=2$ is given in the proof of Theorem \ref{PropN5*}.
  

 We have the following ideal theoretic interpretation of analytic spread zero for a symbolic filtration $\mathcal I=\{\mathfrak p^{(n)}\}$ . We have (by Lemma \ref{sp0}) that 
    $$
    \ell(\mathcal I)=0\mbox{ if and only if }
          \mbox{ for all $n$ and     $f\in \mathfrak p^{(n)}$, there exists $m>0$ such that $f^m\in m_R\mathfrak p^{(mn)}$.}
  $$  
  
 In Theorem \ref{PropN51}, we necessarily have that the symbolic algebra is not finitely generated if $\ell(\mathcal I)<2$ (by Proposition \ref{PropAS11}). A simple example of a symbolic algebra achieving the maximum analytic spread   $\ell(\mathcal I)=2$   may be constructed by taking $\mathfrak p$ to be a regular prime ideal in $R$ ($\mathfrak p=(x,y)$ where $x,y,z$ is a regular system of parameters in $R$). We do not know of an example such that $\ell(\mathcal I)=2$ but the symbolic algebra is not finitely generated.
 
 If $\ell(\mathcal I)<2$, then by Corollary \ref{LemmaN6}, the analytic spread $\ell(\mfp^{(n)})=3$ for all $n>0$ and by 
 Proposition \ref{PropAS10}, we have that $\ell(\mathcal I_a)=3$ for all truncations $\mathcal I_a$ of $\mathcal I$.
  
We  look  a little more closely at the most dramatic case of the theorem, when   $\ell(\mathcal I)=0$. 
  The analytic spread $\ell(\mathcal I)$ being zero has the following interpretation in the geometry of the canonical projection $\phi:\mbox{Proj}(R[\mathcal I])\rightarrow \mbox{Spec}(R)$.
  We have that  
 $\phi^{-1}(\mathfrak p)=\PP^1_{\kappa(\mathfrak p)}$, where $\kappa(\mathfrak p)=(R/\mathfrak p)_{\mathfrak p}$, since $\mbox{Proj}(\oplus_{n\ge 0}\mathfrak p_{\mathfrak p}^n)$ is the blow up of the maximal ideal $\mathfrak p_{\mathfrak p}$
 in the two dimensional regular local ring $R_{\mathfrak p}$, so that $\dim \phi^{-1}(\mathfrak p)=1$, but
 $\phi^{-1}(m_R)=\emptyset$
  since $\ell(\mathcal I)=0$. In particular, the theorem on upper semicontinuity of fiber dimension (\ref{eqN2}) for ideals  fails spectacularly in  this non Noetherian situation.

 \section{Notation}

We will denote the nonnegative integers by $\NN$ and the positive integers by $\ZZ_{>0}$,   the set of nonnegative rational numbers  by $\QQ_{\ge 0}$  and the positive rational numbers by $\QQ_{>0}$. 
We will denote the set of nonnegative real numbers by $\RR_{\ge0}$ and the positive real numbers by $\RR_{>0}$.


A local ring is assumed to be Noetherian.
The maximal ideal of a local ring $R$ will be denoted by $m_R$. 
 Excellent local rings have many excellent properties which are enumerated in \cite[Scholie IV.7.8.3]{EGA}. We will make use of some of these properties without further reference.

\section{ The analytic spread of a filtration}\label{SecFil}

A filtration $\mathcal I=\{I_n\}_{n\in\NN}$ of ideals on a ring $R$ is a descending chain
$$
R=I_0\supset I_1\supset I_2\supset \cdots
$$
of ideals such that $I_iI_j\subset I_{i+j}$ for all $i,j\in \NN$. A filtration $\mathcal I=\{I_n\}$ of ideals on a local ring $(R,\mR)$ is a filtration of $R$ by $\mR$-primary ideals if $I_n$ is $m_R$-primary for $n\ge 1$.
A filtration $\mathcal I=\{I_n\}_{n\in\NN}$ of ideals on a ring $R$ is called a Noetherian filtration if $\bigoplus_{n\ge 0}I_n$ is a finitely generated $R$-algebra.

If $I\subset R$ is an ideal, then $V(I)=\{\mfp\in \mbox{Spec}(R)\mid I\subset \mfp\}$.

For any filtration  $\mathcal I=\{I_n\}$ and $\mfp\in \Spec R,$  let $\mathcal I_\mfp$ denote the filtration $\mathcal I_\mfp=\{I_{n}R_\mfp\}.$

 Let $R$ be a local ring and $\mathcal I=\{I_n\}$ be a filtration of  $R$.   We define the graded $R$-algebra
  $R[\mathcal I]=\sum_{m\ge 0}I_mt^m.$
      
  For the rest of this section , suppose that $R$ is a local ring.
  Let  $\mathcal I=\{I_n\}$ be a filtration of ideals of $R.$  Then, \cite[Lemma 3.1]{CS},
  \begin{equation}\label{dim*}
			\mbox{ For all $n\geq 1,$ $V(I_1)=V(I_n)$ and $\dim R/I_1=\dim R/I_n$.}
			\end{equation}


\begin{Definition} Let  $R$ be a  local ring and $\mathcal I=\{I_n\}$ be a filtration of ideals of $R$.  We define the {\it{dimension of the filtration $\mathcal I$}} to be  $s(\mathcal I)=\dim R/I_n$ (for any $n\geq 1$), and define the height $\mbox{ht}(\mathcal I)$ of $\mathcal I$ to be $\mbox{ht}(\mathcal I)=\mbox{ht}(I_n)$ (for any $n\ge 1$).
\end{Definition}

 The dimension $s(\mathcal I)$ 	 and height $\mbox{ht}(\mathcal I)$ are well-defined by equation (\ref{dim*}). In the case of the trivial filtration $\mathcal I=\{I_n\}$, where $I_n=R$ for all $n$, we have that $s(\mathcal I)=-1$. 


  
  Suppose that $\mathcal I=\{I_n\}$ is a filtration of a local ring $R$. Then the associated graded rings
  $$
  R[\mathcal I]=\sum_{n\ge 0} I_nt^n\mbox{ and }
    S[\mathcal I]=R[\mathcal I][t^{-1}]
  $$
   are subrings of the graded ring $R[t,t^{-1}]$. We have 
  a graded ring 
  $$
  T_{\mathcal I}:=R[\mathcal I]/m_RR[\mathcal I].
  $$
 
	\begin{Definition}\label{trunc}
	Suppose that $\mathcal I=\{I_i\}$  is a filtration of ideals on a local ring $R$. Fix $a\in \ZZ_{>0}.$ The $a$-th {\it truncated  filtration} $\mathcal I_a=\{I_{a,n}\}$ of $\mathcal I$ is defined  by 
	\[ I_{a,n}= \left\{
	\begin{array}{l l}
	~~~~~I_n & \quad \text{if $n\le a$ }\\ \vspace{0.3mm}\\
	\sum\limits_{\substack{i,j>0\\i+j=n}} I_{a,i}I_{a,j} & \quad \text{if $n>a.$ }
	\end{array} \right.\] 	
\end{Definition}

  Let $\mathcal I_a$ be the $a$-th truncation of $\mathcal I$. Then $R[\mathcal I]=\cup_{a\ge 0}R[\mathcal I_a]$ and $S[\mathcal I]=\cup_{a\ge 0}S[\mathcal I_a]$.
  
  The following remark follows from Proposition III.3.2 and Proposition III.3.3 on pages 158 and 159 of \cite{Bo}.

  \begin{Remark} \label{RemAS1} 
Suppose that $\mathcal I$ is a Noetherian filtration. There exists $e>0$ such that for all $m\ge 1$,
$R[\mathcal I]$ is a finitely generated $R[I_{me}t^{me}]$-module. In particular,  $\dim R[\mathcal I]=\dim R[I_{me}t^{me}]$. 
 \end{Remark}

 \begin{Lemma}\label{dim}
  		Let $\mathcal A$ be an $\mathbb N$ or $\mathbb Z$-graded ring. Suppose $\{\mathcal A_a\}_{a\geq 1}$ is a collection of Noetherian graded rings with the same grading as $\mathcal A,$ $\max\{\dim \mathcal A_a:a\ge 1\}<\infty,$ $\mathcal A_{a,n}=\mathcal A_n$ for all $ n\le a$ and for each $a\geq 1$ there is a graded ring homomorphism $\phi_a:\mathcal A_a\rightarrow\mathcal A$ such that $\phi_a(x)=x$ for all homogeneous elements of  $\mathcal A_a$ of degree less than or equal to $a.$ Then $\dim \mathcal A\le\max\{\dim \mathcal A_a:a\ge 1\}.$
  	\end{Lemma}  
  	
  	\begin{proof}
  		Let $P_0\subset P_1\subset \cdots\subset P_r$ be a chain of distinct prime ideals in $\mathcal A$. 	There exist $f_i\in P_i\setminus P_{i-1}$ for $1\le i\le r$. Let $a\in \ZZ_{>0}$ be such that $f_1,\ldots,f_r\in \mathcal A_a$. 
  		Then the prime ideals in the chain of prime ideals in $\mathcal A_a$,
  		$$
  		\phi_a^{-1}(P_0)\subset \phi_a^{-1}(P_1)\subset \cdots\subset  \phi_a^{-1}(P_r)
  		$$
  		are all distinct. Thus $r\le\dim\mathcal A_a$. 
  	\end{proof}
  	
  	\begin{Lemma}\label{LemmaAS2}
	 For any filtration (possibly nonnoetherian) of ideals $\mathcal I$ in $R,$ $\dim R[\mathcal I]\le \dim R+1$, $\dim T_{\mathcal I}\le\dim R$ and $\dim S[\mathcal I]\le\dim R+1$. In particular, if $R$ is a domain  of dimension greater than zero  and $I_1\neq 0$, 
	  then $\dim R[\mathcal I]= \dim R+1$, $\dim S[\mathcal I]=\dim R+1$.
  	\end{Lemma}
  	\begin{proof}
  		Let $\mathcal I_a$ denote the $a$-th truncated filtration of $\mathcal I$ for all $a\geq 1.$ Since  $\mathcal I_a$ is Noetherian for all $a\ge 1$, there exists $d_a\ge 0$ such that 
  		$R[\mathcal I_a]$ is a finitely generated $R[I_{a,d_a}t^{d_a}]$-module by Remark \ref{RemAS1}. Thus $\dim R[I_{a,d_a}t^{d_a}]\le \dim R+1$ and $\dim T_{\mathcal I_a}\le\dim R$ (formula (1) on page 94 of \cite{HS}, \cite[Proposition 5.1.6]{HS}). 
  		
  		Further, by Remark \ref{RemAS1}, there exists $d\ge 0$ such that $S[\mathcal I_a]$ is a finitely generated $R[I_{a,d}t^d,t^{-d}]$-module. By formula (2) on page 94 of \cite{HS}, $\dim R[I_{a,d}t^d,t^{-d}]\le \dim R+1$. Thus $\dim  S[\mathcal I_a]\le \dim R+1$.
  		
  		Since for all $a\geq 1,$ we have maps $\phi_a: R[\mathcal I_a]\rightarrow R[\mathcal I]$ defined by $\phi_a(xt^n)=xt^n$ for all homogeneous $x\in R[\mathcal I_a]$ of degree $n\in \mathbb N,$	$\psi_a: S[\mathcal I_a]\rightarrow S[\mathcal I]$ defined by $\psi_a(xt^n)=xt^n$ for all homogeneous $x\in S[\mathcal I_a]$ of degree $n\in \mathbb Z,$ and $\chi_a:T_{\mathcal I_a}\rightarrow T_{\mathcal I}$ defined by $\chi_a(x+\mR I_{a,n})=x+\mR I_n$ for all homogeneous $x\in T_{\mathcal I_a}$ of degree $n\in\mathbb N,$ we get $\dim R[\mathcal I]\le \dim R+1$, $\dim T_{\mathcal I}\le\dim R$ and $\dim S[\mathcal I]\le\dim R+1$ by Lemma \ref{dim}.
  		
  		Suppose $R$ is domain. Consider the ideal $P=\sum_{n\geqslant 1}I_nt^n\subset R[\mathcal I]$. Then $\height P\geq 1.$  Since $R[\mathcal I]/P\cong R,$ we have $P$ is a prime ideal in $R[\mathcal I].$ Therefore $\dim R[\mathcal I]\ge \dim R+1.$
  		Since $$\dim  S[\mathcal I]\geq \dim S[\mathcal I]_{t^{-1}}=\dim R[t,t^{-1}]=\dim R+1,$$ we have $\dim  S[\mathcal I]\ge \dim R+1.$
  \end{proof}

 

 This allows us to define the analytic spread $\ell(\mathcal I)$ of a filtration $\mathcal I$ by
 \begin{equation}\label{N3}
 \ell(\mathcal I)=\dim T_{\mathcal I}.
 \end{equation}
 This generalizes the classical definition of analytic spread of an ideal $I$, $\ell(I)=\dim T_{I}$
  where $T_{I}=R[It]/m_RR[It]$, since if $\mathcal I$ is the I-adic filtration $\mathcal I=\{I^n\}$, then 
 $T_{\mathcal I}=T_I$, so $\ell(\mathcal I)=\ell(I)$.


  From Lemma \ref{LemmaAS2} we obtain the following lemma.
 
 \begin{Lemma}\label{LemmaN4}
 Suppose that  
   $\mathcal I$ is an arbitrary filtration of a local ring $R$.
 Then 
 $$
 \ell(\mathcal I)\le \dim R,
 $$
 in agreement with the classical bound for ideals $I$, $\ell(I)\le\dim R$. 
 \end{Lemma} 
 
 
 

 
 Suppose that $I$ is an ideal in a local ring. Then we have the inequalities
 \begin{equation}\label{eqAS30}
 \mbox{ht}(I)\le \ell(I)\le \dim R.
 \end{equation}
 (proven for instance in \cite[Corollary 8.3.9]{HS}). An ideal for which the equality $\mbox{ht}(I)=\ell(I)$ holds is called equimultiple. 
 The inequalities (\ref{eqAS30}) continue to hold for Noetherian filtrations.

 \begin{Proposition}\label{PropAS11} Suppose that $\mathcal I$ is a Noetherian filtration in a local ring $R$.
  Then there exists $e>0$ such that
 $\ell(I_{em})=\ell(\mathcal I)$ for all $m>0$. In particular,
 $\mbox{ht}(\mathcal I)\le \ell(\mathcal I)$.
Further, 
  $\mbox{ht}(\mathcal I)\le \ell(\mathcal I)\le \dim R$.
\end{Proposition}
 
 \begin{proof} Let $e>0$ be such that the conclusions of 
 Remark \ref{RemAS1} hold. Then
 $$
 m_RR[\mathcal I]\cap R[I_{em}t^{em}]=m_RR[\mathcal I_{em}t^{em}],
 $$
 so
 $$
 R[I_{em}t^{em}]/m_RR[I_{em}t^{em}]\subset R[\mathcal I]/m_RR[\mathcal I]
 $$
 is a finite inclusion of Noetherian rings, so
 $$
 \dim T_{I_{em}}=\dim R[I_{em}t^{em}]/m_RR[I_{em}t^{em}] = \dim T_{\mathcal I}.
 $$ 
 \end{proof}

 The condition of analytic spread zero has a simple ideal theoretic interpretation.

\begin{Lemma}\label{sp0} Suppose that  $\mathcal I=\{I_n\}$ is a filtration in a local ring $R$. Then the analytic spread
$\ell(\mathcal I)=0$ if and only if
\begin{equation}\label{sp1}
\mbox{For all $n>0$ and $f\in I_n$, there exists $m>0$ such that $f^m\in m_RI_{mn}$.}
\end{equation}
\end{Lemma}

\begin{proof} Let $A=R[\mathcal I]$. We have that $\ell(\mathcal I)=0$ if and only if $\dim A/m_RA=0$ which holds if and only if all minimal prime ideals of $m_RA$ are maximal ideals of $A$. Since $m_RA$ is a homogeneous ideal, all minimal prime ideals  of $m_RA$ are homogeneous (\cite[Lemma 3, page 153]{ZS2}). The only graded maximal ideal of $A$ is $m_R\oplus I_1\oplus I_2\oplus\cdots$. Thus $\ell(\mathcal I)=0$ if and only if $\sqrt{m_RA}=m_R\oplus I_1\oplus I_2\oplus\cdots$, which holds if and only if the condition (\ref{sp1}) holds.
\end{proof}

 \section{Divisorial filtrations}\label{SecDiv}

 Let $R$ be a  local domain of dimension $d$ with quotient field $K$.  Let $\nu$ be a discrete valuation of $K$ with valuation ring $\mathcal O_{\nu}$ and maximal ideal $m_{\nu}$.  Suppose that $R\subset \mathcal O_{\nu}$. Then for $n\in \NN$, define valuation ideals
$$
I(\nu)_n=\{f\in R\mid \nu(f)\ge n\}=m_{\nu}^n\cap R.
$$
 
  A divisorial valuation of $R$ (\cite[Definition 9.3.1]{HS}) is a valuation $\nu$ of $K$ such that if $\mathcal O_{\nu}$ is the valuation ring of $\nu$ with maximal ideal $\mathfrak m_{\nu}$, then $R\subset \mathcal O_{\nu}$ and if $\mfp=\mathfrak m_{\nu}\cap R$ then $\mbox{trdeg}_{\kappa(\mfp)}\kappa(\nu)={\rm ht}(\mfp)-1$, where $\kappa(\mfp)$ is the residue field of $R_{\mfp}$ and $\kappa(\nu)$ is the residue field of $\mathcal O_{\nu}$. If $\nu$ is a divisorial valuation of R such that $m_R=m_{\nu}\cap R$, then $\nu$ is called an $m_R$-valuation.
 
 By \cite[Theorem 9.3.2]{HS}, the valuation ring of every divisorial valuation $\nu$ is Noetherian, hence is a  discrete valuation. 
 Suppose that  $R$ is an excellent local domain. Then a valuation $\nu$ of the quotient field $K$ of $R$ which is nonnegative on $R$ is a divisorial valuation of $R$ if and only if the valuation ring $\mathcal O_{\nu}$  is essentially of finite type over $R$ (\cite[Lemma 6.1]{CS}).




 Suppose that $s\in \NN$.
An $s$-valuation of $R$ is a divisorial valuation of $R$ such that $\dim R/p=s$ where $p=\mm_{\nu}\cap R$. 

An integral  divisorial filtration of $R$ (which we will refer to as a divisorial filtration in this paper) is a filtration $\mathcal I=\{I_m\}$ such that  there exist divisorial valuations $\nu_1,\ldots,\nu_r$ and $a_1,\ldots,a_r\in \ZZ_{\ge 0}$ such that for all $m\in \NN$,
$$
I_m=I(\nu_1)_{ma_1}\cap\cdots\cap I(\nu_r)_{ma_r}.
$$
If $\mathcal I$ is a divisorial filtration, then the ideals $I_m=\overline{I_m}$ are integrally closed for all $m\geq 1$. In fact,
the Rees algebra
$R[\mathcal I]=\sum_{n\ge 0}I_nt^n$ is integrally closed in $R[t]$. This is proven in \cite[Lemma 5.8]{C3}.
\cite[Lemma 5.8]{C3} is stated for divisorial $m_R$-filtrations but the proof is valid for arbitrary divisorial filtrations.

An integral $s$-divisorial filtration of $R$ (which we will refer to as an $s$-divisorial filtration in this paper) is a filtration $\mathcal I=\{I_m\}$ such that  there exist $s$-valuations $\nu_1,\ldots,\nu_r$ and $a_1,\ldots,a_r\in \ZZ_{\ge 0}$ such that for all $m\in \NN$,
\begin{equation}\label{eqDF}
I_m=I(\nu_1)_{ma_1}\cap\cdots\cap I(\nu_r)_{ma_r}.
\end{equation}


   \begin{Theorem}\label{ThmN21}  Suppose that $R$ is a $d$-dimensional excellent local domain and $\mathcal I=\{I_n\}$ is
   a divisorial filtration of $m_R$-primary ideals
    on $R$. Then $\ell(\mathcal I)=d$. 
      \end{Theorem}
   
   \begin{proof} There exist $m_R$-valuations 
   $\nu_1,\ldots,\nu_t$  and $a_1,\ldots, a_t\in \ZZ_{>0}$ such that $\mathcal I=\{I_n\}$ where 
   $I_n=I(\nu_1)_{a_1n}\cap \cdots \cap I(\nu_t)_{a_tn}$ for $n\ge 0$, with
   $I(\nu_i)_{m}=\{f\in R\mid \nu_i(f)\ge m\}$.
   
   Let $S$ be the normalization of $R$ in the quotient field of $R$. Let $\mathfrak m_1,\ldots,\mathfrak m_u$ be the maximal ideals of $S$. Let $J(\nu_i)_m=\{f\in S\mid \nu_i(f)\ge m\}$. For each $i$, there exists $\sigma(i)$ with $1\le \sigma(i)\le u$ such that the ideals $J(\nu_i)_m$ are $\mathfrak m_{\sigma(i)}$-primary for all $m$, and $J(\nu_i)_1=\mathfrak m_{\sigma(i)}$. That is, $\nu_i$ is an $\mathfrak m_{\sigma(i)}$-valuation.
   For $n\in \NN$, let 
   $$
   J_n=J(\nu_1)_{a_1n}\cap \cdots \cap J(\nu_t)_{a_tn}
   $$
   so that $J_n\cap R=I_n$.
   
   Let $\pi:X\rightarrow \mbox{Spec}(S)$ be the blow up of an ideal $K$ such that $K_{\mathfrak m_i}$ is a $(\mathfrak m_i)_{\mathfrak m_i}$-primary ideal  for $1\le i\le u$, $X$ is normal and there exist prime divisors $E_i$ on $X$ such that
   the valuation rings $\mathcal O_{\nu_i}=\mathcal O_{X,E_i}$ for $1\le i\le t$. Let $A$ be the effective Cartier divisor on $X$ such that $\mathcal O_X(-A)=K\mathcal O_X$, so that $-A$ is ample on $X$. Write $A=\sum_{i=1}^sb_iE_i$ where $s\ge t$, $E_1,\ldots,E_s$ are prime Weil divisors with $\mathcal O_{\nu_i}=\mathcal O_{X,E_i}$ for $1\le i\le t$ and $b_i\in \ZZ_{>0}$ for all $i$.
   
   There exists a unique $\alpha\in \QQ_{>0}$ such that $\alpha b_i\ge a_i$ for $1\le i\le t$ and further, there exists an index $i_0$ such that
   $\alpha b_{i_0}=a_{i_0}$. Write $\alpha=\frac{c}{d}$ with $c,d\in \ZZ_{>0}$. Then $mcb_iE_i\ge mda_iE_i$ for $1\le i\le t$ and $mcb_{i_0}E_{i_0}=mda_{i_0}E_{i_0}$ for all $m\geq 0$. Thus
   $$
   \mathcal O_X(-mcA)\subset \mathcal O_X(-\sum_{i=1}^tmda_iE_i)
   $$
   for $m\ge 0$, so that $\Gamma(X,\mathcal O_X(-mcA))\subset J_{md}$ for all $m\in \NN$.

 Since $X$ is normal and $E_{i_0}$ has codimension 1 in $X$, there exists a closed point $q\in E_{i_0}$ such that $E_{i_0}$ is the only irreducible component of $A$ which contains $q$ in its support and $\mathcal O_{X,q}$ and $\mathcal O_{E_{i_0},q}$ are regular local rings.  Let $x_1=0$ be a local equation of $E_{i_0}$ at $q$ and extend $x_1$ to a regular system of parameters $x_1,x_2,\ldots,x_d$ in $\mathcal O_{X,q}$. Let $P_j=(x_1,x_2,\ldots,x_j)$ for $1\le j\le d$. $P_j$ are regular primes in $\mathcal O_{X,q}$ (that is, $\mathcal O_{X,q}/P_j$ is a regular local ring for all $j$). Thus the rule $\omega_j(g)=\mbox{ord}_{P_j}(g)$ for $g\in \mathcal O_{X,q}$ defines a discrete valuation on the quotient field of $R$, which is a $\mathfrak m_{\sigma(i_0)}$-valuation
 (since $E_{i_0}$ is contracted to $m_{\sigma(i_0)}$).
 For $1\le j\le d$ let
 $J(\omega_j)_n=\{f\in S\mid \omega_j(f)\ge n\}$. Then $J(\omega_j)_n=P_j^n\cap S$ for all $n\ge 0$. We have that $\omega_1$ is the valuation $\nu_{i_0}$. Since $x_1,\ldots,x_d$ is a regular system of parameters, 
 \begin{equation}\label{C1}
 P_1^m\cap P_j^{m+1}=P_1^mP_j\mbox{ for all $m\in \NN$.}
 \end{equation}
  Let $Z_j$ be the closed subvariety of $X$ such that its ideal sheaf satisfies $(\mathcal I_{Z_j})_q=P_j$ for $1\le j\le d$. Then $Z_1=E_{i_0}$, $Z_d=q$, $\dim Z_j=d-j$ for all $j$ and $\pi(Z_j)=\mathfrak m_{\sigma(i)}$ for all $j$. For $m\ge 0$ and $1\le j\le d$, we have 
  $$
  \begin{array}{lll}
  (\mathcal I_{Z_j}\otimes \mathcal O_X(-mcA))_q
  &=&(\mathcal I_{Z_j}\otimes \mathcal O_X(-mcb_{i_0}E_{i_0}))_q  
   =(\mathcal I_{Z_j}\otimes \mathcal O_X(-mda_{i_0}E_{i_0}))_q \\ 
     &=&P_1^{mda_{i_0}}P_j=P_1^{mda_{i_0}}\cap P_j^{mda_{i_0}+1}.
     \end{array}
  $$
   Observe that  we have inclusions of sheaves
  $$
  \mathcal I_{Z_1}\otimes \mathcal O_X(-mcA)\subset \mathcal I_{Z_2}\otimes \mathcal O_X(-mcA)\subset \cdots\subset
  \mathcal I_{Z_d}\otimes\mathcal O_X(-mcA)\subset \mathcal O_X(-mcA)\subset \mathcal O_X.
  $$
  Since $-A$ is ample, for $m\gg 0$, $\mathcal I_{Z_j}\otimes\mathcal O_X(-mcA)$ is generated by global sections for $1\le j\le d$, so that
  \begin{equation}\label{C2}
  \begin{array}{lll}
  \Gamma(X,\mathcal I_{Z_j}\otimes\mathcal O_X(-mcA))\mathcal O_{X,q}&=&P_1^{mda_{i_0}}\cap P_j^{mda_{i_0}+1}\mbox{ and}\\
  \Gamma(X,\mathcal O_X(-mcA))\mathcal O_{X,q}&=&P_1^{mda_{i_0}}.
  \end{array}
  \end{equation}
  We have inclusions 
  $$
  \begin{array}{l}
  \Gamma(X,\mathcal I_{Z_1}\otimes \mathcal O_X(-mcA))\subset \Gamma(X,\mathcal I_{Z_2}\otimes \mathcal O_X(-mcA))\subset\\
  \cdots\subset \Gamma(X,\mathcal I_{Z_d}\otimes\mathcal O_X(-mcA))\subset \Gamma(X,\mathcal O_X(-mcA))\subset J_{md}.
  \end{array}
  $$
    By (\ref{C2}), for $1\le j\le d-1$, there exists 
  $f_j\in \Gamma(X,\mathcal I_{Z_{j+1}}\otimes \mathcal O_X(-mcA))\subset J_{md}$  such that $f_j\in P_{j+1}^{mda_{i_0}+1}$ but $f_j\not\in P_j^{mda_{i_0}+1}$,  so that
  $f_j\in 
  P_{j+1}^{mda_{i_0}+1}\cap J_{md}=J(\omega_{j+1})_{mda_{i_0}+1}\cap J_{md}$, but $f_j\not\in P_j^{mda_{i_0}+1}\cap J_{md}=J(\omega_j)_{mda_{i_0}+1}\cap J_{md}$ and there exists $f_{d}\in J_{dm}$ such that $f_{d}\not\in J(\omega_d)_{mda_{i_0}+1}\cap J_{dm}$. 
  
  Let $B=\oplus_{n\ge 0}J_n$, which is a graded ring. 
  Let 
  $$
  C_j=\oplus_{m\ge 0}J(\omega_{j+1})_{a_{i_0}m+1}\cap J_m
  $$
   for $0\le j\le d-1$ and 
   $$
   C_d=\mathfrak m_{\sigma(i_0)}\oplus J_1\oplus J_2\oplus\cdots.
   $$
 We will now show that the ideals $C_j$ are prime ideals in $B$. First observe that none of the  $C_j$ are equal to $B$ since 
 $C_j\cap S=J(\omega_{j+1})_1=\mathfrak m_{\sigma(i_0)}$ for $1\le j\le d-1$ and $C_d\cap S=\mathfrak m_{\sigma(i_0)}$.
 Suppose $1\le j\le d-1$ and $f\in J_m$, $g\in J_n$ are such that $fg\in J(\omega_{j+1})_{a_{i_0}(m+n)+1}$. 
 Then $\omega_{j+1}(fg)\ge a_{i_0}(m+n)+1$. We have that $J_m\subset J(\nu_{i_0})_{a_{i_0}m}\subset J(\omega_{j+1})_{a_{i_0}m}$ so that $\omega_{j+1}(f)\ge a_{i_0}m$. Similarly, $\omega_{j+1}(g)\ge a_{i_0}n$. Thus
 either $\omega_{j+1}(f)\ge a_{i_0}m+1$ or $\omega_{j+1}(g)\ge a_{i_0}n+1$, so that $f\in J(\omega_{j+1})_{a_{i_0}m+1}\cap J_m$ or $g\in J(\omega_{j+1})_{a_{i_0}n+1}\cap J_n$.   Thus the $C_j$ are prime ideals. 
 
 We found $f_j\in C_j\setminus C_{j-1}$ for $1\le j\le d$. Thus
 $$
 C_0\subset C_1\subset C_2\subset \cdots\subset C_d
 $$
 is a chain of distinct prime ideals in $B$.
 
 There is a natural inclusion of graded rings $R[\mathcal I]=\oplus_{n\ge 0}I_n\subset B=\oplus_{n\ge 0}J_n$. We will now show that $B$ is integral over $R[\mathcal I]$. For $a\in \ZZ_{>0}$, let $R[\mathcal I]_a$ be the $a$-th truncation of $R[\mathcal I]$ and $B_a$ be the $a$-th truncation of $B$, so that $R[\mathcal I]_a$ is the subalgebra of $R[\mathcal I]$ generated by $\oplus_{n\le a}I_n$ and $B_a$ is the subalgebra of $B$ generated by $\oplus_{n\le a}J_n$. It suffices to show that homogeneous elements of $B$ are integral over $R[\mathcal I]$.  Suppose that $f\in J_a$ for some $a$.
 Then $f\in B_a$. Let $0\ne x$ be in the conductor of $S$ over $R$. Then $xJ_n\subset I_n$ for all $n$ since $I_n=J_n\cap R$.
Thus $xB_a\subset R[\mathcal I]_a$, so $f^i\in \frac{1}{x}R[\mathcal I]_a$ for all $i\in \NN$, and so the algebra 
$R[\mathcal I]_a[f]\subset \frac{1}{x}R[\mathcal I]_a$. Since $\frac{1}{x}R[\mathcal I]_a$ is a finitely generated $R[\mathcal I]_a$-module and $R[\mathcal I]_a$ is a Noetherian ring, the ring $R[\mathcal I]_a[f]$ is a finitely generated $R[\mathcal I]_a$-module, so that $f$ is integral over $R[\mathcal I]_a$.

We have a chain of prime ideals 
$$
Q_0\subset Q_1\subset Q_2\subset \cdots\subset Q_d
$$
in $R[\mathcal I]$ where $Q_i:=C_i\cap R[\mathcal I]$. The $Q_i$ are all distinct since the $C_i$ are all distinct and $B$ is integral over $R[\mathcal I]$ (by \cite[Theorem A.6 (b)]{BH}). It remains to show that $m_RR[\mathcal I]\subset Q_0$, so that 
$\dim R[\mathcal I]/m_RR[\mathcal I]\ge d$. Since this is the maximum possible dimension of $R[\mathcal I]/m_RR[\mathcal I]$ by Lemma \ref{LemmaAS2}, we have that $\ell(\mathcal I)=d$.

We now show that $m_RR[\mathcal I]\subset Q_0$. First we observe that if $g\in m_R$ then $\nu_{i_0}(g)\ge 1$ since $\nu_{i_0}$ is an $m_R$-valuation. Suppose that $f\in m_RI_n$. Then $f=\sum g_kf_k$ with $g_k\in m_R$ and $f_k\in I_n$. Thus $\nu_{i_0}(g_jf_j)\ge  na_{i_0}+1$ for all $j$ so that $f\in I(\nu_{i_0})_{na_{i_0}+1}$ and thus $f\in I(\nu_{i_0})_{na_{i_0}+1}\cap I_n$. Since $\omega_1$ is the valuation $\nu_{i_0}$, we have $m_RR[\mathcal I]\subset Q_0$.

   \end{proof}

\begin{Proposition}{\label{persistence}}
	Let $R$ be a local domain and $\mathcal I=\{I_n\}$ be a divisorial filtration in $R$ where $I_n=\bigcap_{i=1}^rI(\nu_i)_{na_i}$ for all $n\geq 1.$ Let $\mR\in\Ass(R/I_1)$. Then $\mR\in\Ass(R/\overline{I_{a,n}})$ for all $a, n\geq 1$ where $\mathcal I_a=\{I_{a,n}\}$ is the $a$-th truncated filtration of $\mathcal I.$ In particular, $m_R\in \mbox{Ass}(R/I_n)$ for all $n\ge 1$.
\end{Proposition}
\begin{proof}
	Suppose $\Ass(R/I_1)=\{\mR\}.$ Since $\Min\Ass(R/I_1)=\Min\Ass(R/I_{a,n})$ for all $a,n\geq 1$, we have $\mR\in\Ass(R/\overline{I_{a,n}})$ for all $a, n\geq 1$.
	
	Suppose the cardinality of $\Ass(R/I_1)$ is greater than one and $\mR\in\Ass(R/I_1)$. Without loss of generality let us assume that the centers $I(\nu_i)_1$ of $\nu_i$ on $R$ are $m_R$ for $1\le i\le c$ and the centers $I(\nu_i)_1$ are not $m_R$ for $i>c $.
	
	Fix $a.$ Since $\mR\in\Ass(R/I_{1})$, there exists $y\in R\setminus I_1$ such that $\mR y\in I_1$. Therefore $$y\in I_1: \mR^{\infty}=\bigcap_{ j> c}^rI(\nu_j)_{a_j}.$$ Thus $y\notin \cap_{i=1}^cI(\nu_i)_{a_i}.$ Hence $y^{n}\in \bigcap_{j>c}^rI(\nu_j)_{na_j}\setminus \cap_{i=1}^cI(\nu_i)_{na_i}$ for all $n\geq 1$. Therefore $y^n\notin I_n$ for all $n\geq 1$. Since $I_{a,n}\subset I_n= {\overline{I_n}}$, we have $y^n\notin {\overline{I_{a,n}}}$ for all $n\geq 1$.
	
	Let $v\in \mR^{b}$  where $b=a_1+\cdots+a_c$. Then $\nu_i(yv)\geq a_i$ for all $1\leq i\leq c$. Thus $yv\in \cap_{i=1}^cI(\nu_i)_{a_i}.$ Since $y\in \bigcap_{j>c}^rI(\nu_j)_{a_j}$, we have $yv\in I_1$ and hence for all $n\geq 1$, 
	$$
	y^{n}\mR^{nb}\subset I_1^{n}=I_{a,1}^{n}\subset I_{a,n}\subset{\overline {I_{a,n}}}.
	$$
	 Let $m\geq 1$ be an integer such that $y^n\mR^m\subseteq{\overline {I_{a,n}}}$ and $y^n\mR^{m-1}\nsubseteq{\overline {I_{a,n}}}$. Let $x\in \mR^{m-1}\setminus\mR^m$ such that $y^nx\notin {\overline {I_{a,n}}}.$	Then $\mR=({\overline {I_{a,n}}}:_R y^nx).$ Therefore
	$\mR\in\Ass(R/\overline{I_{a,n}})$ for all $n\geq 1$. 
\end{proof}

 \begin{Lemma}\label{valuationlemma2} Suppose that $R$ is a local domain and $\mathcal I=\{I_n\}$ is a divisorial filtration on $R$. Suppose that $P$ is a prime ideal of $R$ and there exists $t\in \ZZ_{>0}$ such that $P\in \mbox{Ass}(R/I_t)$. Then there exists $n_0\in \ZZ_{>0}$ such that $P\in \mbox{Ass}(R/I_n)$ for all $n\ge n_0$.
 \end{Lemma}
 
 \begin{proof} Let $I_n=I(\nu_1)_{a_1n}\cap\cdots\cap I(\nu_r)_{a_rn}$ for $n\in \NN$. By Lemma 3, page 343 of Zariski Samuel Vol. II, for all $m\in \ZZ_{>0}$, the ideal $I(\nu_i)_m$ is $P_i$-primary, where the prime ideal $P_i=I(\nu_i)_1$ is the center of $\nu_i$ on $R$. Let $P_1,\ldots,P_s$ be the distinct centers of the $\nu_i$ on $R$ for $1\le i\le r$. For $k$ with $1\le k\le s$ and $n\in \ZZ_{>0}$, let
 	$$
 	Q(k)_n=\bigcap\limits_{I(\nu_i)_1=P_k}I(\nu_i)_{a_in},
 	$$
 	which is a $P_k$-primary ideal.  
 	Thus for all  $1\le k\le s$, $P_k\in \mbox{Ass}(R/I_n)$ if and only if 
 	$I_n\ne \bigcap\limits_{\substack{1\leq i\leq s\\ i\ne k}}Q(i)_n$.
 	
 	Suppose that $P\in \mbox{Ass}(R/I_t)$. Then $P=P_k$ for some $k$. After reindexing the $\nu_i$, there exists  $c>0$ such that the centers $I(\nu_i)_1=P$ if $1\le i\le c$ and $I(\nu_i)_1\ne P$ if $c<i$.	
 	
 	Thus 
 	there exists $f\in \cap_{i>c}I(\nu_i)_{a_it}\setminus I_t$. Therefore $\nu_i(f)\ge a_it$ for $i>c$ and there exists $j$ with $1\le j\le c$ such that $\nu_j(f)\le a_jt-1$. Let $0\ne g\in I_1$ be arbitrary. Then $\nu(g)\ge a_i$ for all $i$. Let $\beta=\nu_j(g)\ge a_j$.
 	
 	Let $n\in \NN$. Write $n=mt+s$ with $m\in \NN$ and $0\le s<t$.
 	$\nu_i(f^mg^s)\ge na_i$ for $i>c$ and 
 	$$
 	\nu_j(f^mg^s)\le m(a_jt-1)+s\beta=(mt+s)a_j+s(\beta-a_j)-m
 	=na_j+s(\beta-a_j)-m<na_j
 	$$
 	for $m>s(\beta-a_j)$. Thus for $m>s(\beta-a_j)$, we have that $f^mg^s\in \cap_{i>c}I(\nu_i)_{a_in}\setminus I_n$ which implies that
 	$P\in \mbox{Ass}(R/I_n)$.
 \end{proof}

 Suppose that $R$ is  a local domain and $\mathcal I=\{I_n\}$ is a divisorial filtration of $R$ where
 $I_n=I(\nu_1)_{a_1n}\cap \cdots\cap I(\nu_r)_{a_rn}$.

 Let $S=S[\mathcal I]$.   Let $I_n=R$ for $n\le 0$. Then for $r\in \ZZ_{>0}$.
 \begin{equation}\label{eqtminus}
 t^{-r}S=\sum_{n\in \ZZ}I_{n+r}t^n.
 \end{equation}
 
\begin{Lemma}\label{LemmaAS1} Let $K$ be an ideal in $R$ such that $I_1\subset K.$ Suppose that $n\in\NN$. Then there exists $r\in \ZZ_{>0}$ such that 
	$(I_{n+1})^r\subset KI_{rn}$. In particular the ideal $\oplus_{n\ge 0}I_{n+1}t^n\subset \sqrt{KR[\mathcal I]}$.
\end{Lemma}

\begin{proof} For all $r\in \ZZ_{>0}$, $I_{n+1}^r= I_{n+1} I_{n+1}^{r-1}\subset K I_{n+1}^{r-1}$. Note that $ I_{n+1}^{r-1}\subset  I_{(n+1)(r-1)}.$ Thus if $r\ge n+1$, then $I_{n+1}^{r-1}\subset I_{rn}$.
\end{proof}

\begin{Lemma}\label{Lemmagen} Let $R$ be a local domain and $\mathcal I=\{I_n\}$  be a divisorial filtration of ideals in $R$, where $I_n=I(\nu_1)_{a_1n}\cap\cdots\cap I(\nu_r)_{a_rn}$. 
For $1\le i\le r$, let 
$$
P_i=\sum_{n\in \ZZ}I(\nu_i)_{a_in+1}\cap I_nt^n.
$$
 Then $P_i$ is a prime ideal in $S=S[\mathcal I]$.
Let 
$$
Q_i=\sum_{n\in \ZZ}I(\nu_i)_{a_i(n+1)}\cap I_nt^n.
$$
 Then $Q_i$ is $P_i$-primary for $1\le i\le r$ and 
$$
t^{-1}S=Q_1\cap\cdots\cap Q_r.
$$
\end{Lemma}

\begin{proof}
Observe that if  $n\le 0$, then the valuation ideal $I(\nu_i)_n=R$, and so $I_n=R$ for $n\le 0$.  Since $P_i$ is a graded $R$-module, to show that it is an ideal in $S$, it suffices to show that if $f\in I(\nu_i)_{a_ia+1}\cap I_a$ and $g\in  I_b$ then 
$fg\in I(\nu_i)_{a_i(a+b)+1}\cap I_{a+b}$. This follows since $\nu_i(fg)=\nu_i(f)+\nu_i(g)\ge (a_ia+1)+a_ib\ge a_i(a+b)+1$
so $fg\in I(\nu_i)_{a_i(a+b)+1}$. $P_i\ne R$ since $R\cap P_i=I(\nu_i)_1$ (which is  a prime ideal). Since $P_i$ is graded, to show that $P_i$ is a prime ideal, it suffices to show that if $f\in I_a$ and $g\in I_b$ are such that 
$fg\in I(\nu_i)_{a_i(a+b)+1}\cap I_{a+b}$, then either $f\in I(\nu_i)_{a_ia+1}\cap I_a$ or $g\in I(\nu)_{a_ib+1}\cap I_b$. 
This follows since    $\nu_i(f)\ge a_ia$, $\nu_i(g)\ge a_ib$ and $\nu_i(fg)=\nu_i(f)+\nu_i(g)\ge a_i(a+b)+1$ so either $\nu_i(f)\ge a_ia+1$ or $\nu_i(g)\ge a_ib+1$.

We now show that $Q_i$ is a primary ideal. It suffices to show that if $f\in I_a$, $g\in I_b$, $fg\in I(\nu_i)_{a_i(a+b+1)}\cap I_{a+b}$ and $f\not\in I(\nu_i)_{a_i(a+1)}\cap I_a$, then there exists an $m>0$ such that $g^m\in I(\nu_i)_{a_i(mb+1)}\cap I_{mb}$.
With these assumptions we have that $\nu_i(f)<a_i(a+1)$ and $\nu_i(fg)\ge a_i(a+b+1)$ so that $\nu_i(g)>a_ib$,  and thus  $\nu_i(g)=a_ib+c$ for some $c>0$. There exists $m>0$ such that $mc\ge a_i$. Thus $\nu_i(g^m)=ma_ib+mc\ge a_i(mb+1)$ so that $g^m\in I(\nu_i)_{a_i(mb+1)}\cap I_{mb}$.

We now show that $\sqrt{Q_i}=P_i$. $Q_i\subset P_i$ since $a_i(n+1)\ge a_in+1$ for all $i$ and $n\ge 0$. We then have that 
$\sqrt{Q_i}=P_i$ since $f\in I(\nu_i)_{a_in+1}\cap I_n$ implies $f^m\in I(\nu_i)_{a_i(mn+1)}\cap I_{mn}$ for $m\ge a_i$.

By (\ref{eqtminus}), $t^{-1}S=\sum_{n\in \ZZ}I_{n+1}t^n=Q_1\cap\cdots\cap Q_r$.
 \end{proof}

 \begin{Theorem}\label{PropAS7} Suppose that $R$ is a  local domain and $\mathcal I=\{I_n\}$ is a divisorial filtration on $R$. Let $I_n=I(\nu_1)_{a_1n}\cap\cdots\cap I(\nu_r)_{a_rn}$ for $n\ge 1$, some valuations $\nu_i$ and some $a_1,\ldots,a_r\in \ZZ_{>0}$. Suppose that $\ell(\mathcal I)=\dim R$. Then for some $\nu_i$, the center $m_{\nu_i}\cap R=\{f\in R\mid \nu_i(f)>0\}$   is $m_R$. There exists a positive integer $n_0$ such that $m_R$ is an associated prime of $I_n=\overline{I_n}$ for all $n\geq n_0$.
	\end{Theorem}

 \begin{proof} Let $S=S[\mathcal I]$ and let notation be as in Lemma \ref{Lemmagen}.
  Let $J$ be the graded ideal 
 \begin{equation}\label{eqt1}
 J=\sum_{n\ge 0}I_{n+1}t^n\subset R[\mathcal I].
 \end{equation}  
  By assumption, there exists a prime ideal $\overline U$ of $T_{\mathcal I}$ such that 
$\dim T_{\mathcal I}/\overline U=\dim R$.
 We have isomorphisms of graded $R$-algebras  
 $$
 A:=S/(t^{-1}S+m_RS)\cong \sum_{n\ge 0} I_{n}/(I_{n+1}+m_RI_{n})t^n\cong R[\mathcal I]/(J+m_RR[\mathcal I]).
 $$
 By Lemma \ref{LemmaAS1}, the nilradical of  $R[\mathcal I]/(J+m_RR[\mathcal I])$ is $\sqrt{m_RR[\mathcal I]}/(J+m_RR[\mathcal I])$. Thus
  the quotient of $T_{\mathcal I}$ by its nilradical is isomorphic as a graded $R$-algebra to the quotient of $A$
 by its nilradical, and so there exists a prime ideal $U'$ of $A$ such that $\dim A/U'=\dim R$.  Let $U$ be the preimage of $U'$ in $S$. We have that $t^{-1}S+m_RS\subset U$  and $t^{-1}\ne 0$ in the domain $S$ so that $\mbox{ht}(U)\ge 1$. Since $\dim S/U=\dim R$, we have that
  $$
  1+\dim R\le \dim S/U+\mbox{ht}(U)\le \dim S\le \dim R+1
  $$
  by Lemma \ref{LemmaAS2} so $\dim S=\dim R+1$ and $\mbox{ht}(U)=1$.
  We further have that $U\cap R=m_R$, since $m_RS\subset U$. Now $\sqrt{t^{-1}S}=\cap_{i=1}^r P_i\subset U$ so that $P_i\subset U$ for some $i$. Thus $P_i=U$ since $\mbox{ht}(U)=1$, and so $m_R=P_i\cap R=I(\nu_i)_1=m_{\nu_i}\cap R$ is the center of $\nu_i$ on $R$. 
  
  We will now show that $m_R$ is an associated prime of some $I_n$. Suppose that $m_R$ is not an associated prime of any $I_n$. We will derive a contradiction. After reindexing, we may suppose that, for some $s$, $\nu_i$ is an  $m_R$-valuation for $i\le s$  and $\nu_i$ is not an $m_R$-valuation for $i>s$.  Since $m_R$ is not an associated prime of $I_n$ for all $n$, we thus have that $I_n=I(\nu_{s+1})_{a_{s+1}n}\cap \cdots \cap I(\nu_{r})_{a_rn}$ for all $n$. Since none of  $\nu_{s+1},\ldots,\nu_r$ is an $m_R$-valuation, we have that $\ell(\mathcal I)<\dim R$ by the first part of this proof, a contradiction.
  Thus there is some positive integer $n_0$ such that $m_R$ is an associated prime of $I_{n_0}$. Thus $m_R$ is an associated prime of $I_{n}$ for all $n\gg 0$ by Lemma \ref{valuationlemma2}.
    
  \end{proof}
  
  \begin{Remark} Theorem \ref{PropAS7} shows that if $\ell(\mathcal I)=\dim R$ then one of the prime ideals $P_i$ of  Lemma \ref{Lemmagen} is a height one prime ideal in $S[\mathcal I]$ such that $P_i\cap R=m_R$.
  \end{Remark}
  
  \begin{Remark} The proofs of Lemma \ref{Lemmagen} and \ref{PropAS7} prove the following more general statement. Let $R$ be a  local ring and $\mathcal J(i)=\{J(i)_n\}_{n\in\NN}$ be filtrations of ideals in $R$ with $J(i)_1\subsetneq R$ for all $1\le i\le r.$ Suppose $\bigcap_{n\ge 1}J(i)_n=0$ and $\mathcal G_i=\bigoplus_{n\ge 0}J(i)_n/J(i)_{n+1}$ are domains for all $1\le i\le r.$ 
		
		Consider the filtration $\mathcal J=\{J_n=J(1)_{a_1n}\cap\cdots\cap J(r)_{a_rn}\}$ for some fixed $a_1,\ldots,a_r\in \ZZ_{>0}$. For $1\le i\le r$, let $P_i=\sum_{n\in \ZZ}J(i)_{a_in+1}\cap J_nt^n$. Then $P_i$ is a prime ideal in $S[\mathcal J]$.
		Let $Q_i=\sum_{n\in \ZZ}J(i)_{a_i(n+1)}\cap J_nt^n$. Then $Q_i$ is $P_i$-primary for $1\le i\le r$, and 
		$$
		t^{-1}S[\mathcal J]=Q_1\cap\cdots\cap Q_r.
		$$
		Suppose that $\ell(\mathcal J)=\dim R$. Then there exists a prime ideal $P_i=\sum_{n\in \ZZ}J(i)_{a_in+1}\cap J_nt^n$ in $S[\mathcal J]$ for some $i\in\{1,\ldots,r\}$ such that $\height P_i=1$ and $P_i\cap R=\mR.$		\end{Remark}

\begin{Corollary}\label{CorN31} Suppose that $R$ is a local domain and $\mathcal I=\{I_n\}$ is an $s$-divisorial filtration on $R$. Then $\ell(\mathcal I_Q)<\dim(R_Q)$ for all prime ideals $Q$ of $R$ which are not minimal primes of $I_1$.
\end{Corollary}

  \begin{Corollary}\label{CorAS4} Suppose that $R$ is a local domain and $\mathcal I$ is an $s$-divisorial filtration on $R$ with $s\ge 1$. Then $\ell(\mathcal I)<\dim R$.
  \end{Corollary}



\begin{Corollary}{\label{max}}
	Let $R$ be a local domain and $\mathcal I=\{I_n\}$ be a divisorial filtration in $R$ where $I_n=\bigcap_{i=1}^rI(\nu_i)_{na_i}$ for all $n\geq 1.$ Suppose $\mR\in\Ass(R/I_1)$. Then $\ell(\mathcal I_a)=\dim R$ for all $a$-th truncated filtration of $\mathcal I$ and hence $\ell(\mathcal I)\leq \ell(\mathcal I_a)$ for all $a\geq 1.$
\end{Corollary}

\begin{proof}
	By Proposition \ref{persistence}, $\mR\in\Ass(R/\overline{I_{a,n}})$ for all $n, a\geq 1.$ Since $\mathcal I_a$ is a Noetherian filtration, there exists an integer $m$ such that $\ell(\mathcal I_a)=\ell(I_{a,m}).$ Therefore by \cite{Mc} ,\cite[Theorem 5.4.6]{HS},
	Theorem \ref{ThmN10}, we have $\ell(\mathcal I_a)=\dim R.$
\end{proof}

\begin{Corollary}\label{SmSpread}
	Let $R$ be a local domain and $\mathcal I=\{I_n\}$ be a divisorial filtration in $R$ where $I_n=\bigcap_{i=1}^rI(\nu_i)_{na_i}$ for all $n\geq 1.$ Suppose $I_1=\cap_{i=1}^rI(\nu_i)_{a_i}$ is a minimal primary decomposition of $I_1$ and $\ell(\mathcal I_a)<\dim R$ for some $a\geq 1$ where $\mathcal I_a$ is the $a$-th truncated filtration of $\mathcal I.$ Then $\ell(\mathcal I)<\dim R.$
\end{Corollary}
\begin{proof}
	If $\ell(\mathcal I)=\dim R$ then by Theorem \ref{PropAS7}, $m_{\nu_i}\cap R=\mR$ for some $i.$ Thus $\mR\in\Ass(R/I_1).$ Therefore by Corollary \ref{max},  we get $\ell(\mathcal I_a)=\dim R$ which is a contradiction.
\end{proof}
Let $R$ be a local ring and $I$ an ideal in $R$. In \cite{Br} Brodmann proved that  $\ell(\mathcal I)\leq \dim R-\liminf_n \depth R/I^n$ where $\mathcal I=\{I^n\}$. If $R$ has infinite residue field then Burch improved the result of Brodmann for the filtration $\mathcal I=\{\overline{I^n}\}$ and proved that  $\ell(\mathcal I)\leq \dim R-\liminf_n \depth R/\overline{I^n}$ \cite{Bu}. This result was generalized to the filtration $\mathcal I=\{I^{(n)}\}$ if the Symbolic Rees algebra of $I$ is finitely generated \cite{BV}.  We generalize Burch's result for divisorial filtrations under some extra conditions.
	\begin{Corollary}{\em(Burch's inequality for divisorial filtration)} Suppose $R$ is a local domain and $\mathcal I=\{I_n\}$ is a divisorial filtration in $R$. Suppose one of the following holds.
	\begin{itemize}
		\item[$(i)$] $\mR\in\Ass(R/I_t)$ for some $t\geq 1$.
		\item[$(ii)$] The filtration $\mathcal I$ is an $1$-divisorial filtration. 
	\end{itemize}		
	Then $\ell(\mathcal I)\leq \dim R-\liminf_n \depth R/I_n.$
\end{Corollary}
\begin{proof}
	$(i)$ By Lemma \ref{valuationlemma2},  there exists a positive integer $n_0$ such that $\mR\in\Ass(R/I_n)$ for all $n\geq n_0$. Thus $\liminf_n \depth R/I_n=0.$ Now by Lemma \ref{LemmaAS2}, $\ell(\mathcal I)\leq \dim R.$
	 	
	$(ii)$	If $\dim R\leq 1$ then $I_n=0$  for all $n\geq 1$ and hence $0=\ell(\mathcal I)\leq \dim R-\liminf_n \depth R/I_n.$ Suppose $\dim R\geq 2.$ Then by Corollary \ref{CorAS4}, we have $\ell(\mathcal I)\le\dim R-1.$ Since $\mathcal I$ is a $1$-divisorial filtration, we have $\depth R/I_n\geq 1$ and $\dim R/I_n=1$ for all $n\geq 1.$ Thus $\depth R/I_n= 1$ for all $n\geq 1.$ Therefore $\ell(\mathcal I)\leq \dim R-\liminf_n \depth R/I_n.$	
\end{proof}

  \section{Symbolic Algebras}\label{SecSym}

 Suppose that $I\subset J$ are proper ideals in a local ring $R$. Define
 $S_J(I)=\oplus_{k\ge 0}I^k:J^{\infty}$ where $I^k:J^{\infty}=\cup_{i=1}^{\infty}I^k:_RJ^i$. 

\begin{Theorem}(\cite[Theorem 2.6]{CHS})
\label{strengthen}
Let $(R,\mm)$ be an excellent domain, and let $I$ and $J$ be proper ideals of $R$. Then the following conditions are equivalent:
\begin{enumerate}
\item[{\em (a)}]  $S_J(I)$ is finitely generated.
\item[{\em (b)}]  There exists an integer $r>0$ such that $\ell((I^r: J^\infty)_P)<\dim R_P$ for all $P\in V(J)$.
\end{enumerate}
\end{Theorem}

 A  related result  was proven by Katz and Ratliff in Theorem A and Corollary~1 of \cite{KR}.
 
 Let $I$ be an ideal in a local ring $R$. For $n\in \ZZ_{>0}$, the $n$-the symbolic power $I^{(n)}$ of $I$ is
 $$
 I^{(n)}=\cap_{\mfp \in \mbox{Ass}(R/I)}(I^nR_{\mfp}\cap R).
 $$
 Let $J$ be the intersection of all asymptotic prime divisors of $I$ which are not minimal primes. Then
 $I^{(n)}=I^n:J^{\infty}$ and the symbolic algebra $\oplus_{n\ge 0}I^{(n)}=S_J(I)$.
 


   \begin{Corollary}\label{LemmaN6} Suppose that $R$ is an excellent local domain of dimension $d$ and $I=P_1\cap \cdots \cap P_r$ is an intersection of
    prime ideals $P_i$ of $R$ of  a common height. Then the ring $\bigoplus_{n\ge 0}I^{(n)}$ is a finitely generated $R$-algebra if and only if there exists $n\in \ZZ_{>0}$ such that the analytic spread $\ell( I^{(n)}_Q)<\mbox{ht}(Q)$ for all prime ideals $Q$ of $R$ which contain $I$ and are not one of the minimal primes $P_i$ of $I$. 
   \end{Corollary}

 With our assumption that $R$ is a local ring of dimension $d$ and $I=P_1\cap \cdots \cap P_r$ is an intersection of prime ideals $P_i$ of $R$ of  a common height in Corollary \ref{LemmaN6}, we have that  $\ell( I^{(n)}_{P_i})=\mbox{ht}(P_i)=\mbox{ht}(I)$  for the minimal primes $P_i$ of $I$.

 \begin{Corollary}\label{CorN70} Suppose that $R$ is an excellent local domain of dimension $d$ and $I=P_1\cap \cdots \cap P_r$ is an intersection of
    prime ideals $P_i$ of $R$ of  a common height. If $I^{(n)}$ is equimultiple for some $n$ then the symbolic algebra $\oplus_{n\ge 0}I^{(n)}$ is a finitely generated $R$-algebra.
    \end{Corollary}

 \begin{proof}    If $I^{(n)}$ is equimultiple, then by (\ref{eqN7*}), 
    $\ell(I^{(n)}_Q)=\mbox{ht}(I^{(n)}_Q)$ for all prime ideals $Q$ containing $I$, so that if $Q$ is not a minimal prime of $I$, we have that $\ell(I^{(n)}_Q)=\mbox{ht}(I_Q)<\dim R_Q$, and so $I^{(n)}$ satisfies the criterion of Corollary \ref{LemmaN6}.
    \end{proof}
    
    However, there exist ideals $I$ such that the symbolic algebra $\oplus_{n\ge 0}I^{(n)}$ is a finitely generated $R$-algebra but no symbolic power $I^{(n)}$ is equimultiple,  as is shown in the following example.

    \begin{Example}\label{ExAS1}(\cite[Example 8.4]{CS}) There exists a height one prime ideal $P$ in a normal, excellent 3 dimensional local ring $R$
    such that no symbolic power of $P$ is equimultiple but the symbolic algebra $\oplus_{n\ge 0}P^{(n)}$ is a finitely generated $R$-algebra.
    \end{Example}

  
  
 
  In contrast to the conclusions of Corollary \ref{LemmaN6}, we have that  inequality of analytic spread and height    $\ell(\mathcal I_Q)<\mbox{ht}(Q)$ holds at all non minimal primes for  symbolic filtrations, irregardless of whether their symbolic algebra is a finitely generated $R$-algebra. 
 The following proposition follows from Corollary \ref{CorN31}.

   \begin{Proposition}\label{PropN22} Suppose that $R$ is a local domain of dimension $d$ and $I=P_1\cap \cdots \cap P_r$ is an intersection of
    prime ideals $P_i$ of $R$ of  a common positive height. Suppose $R_{P_i}$ is a regular local ring for $1\leq i\leq r$. Let $\mathcal I=\{I^{(n)}\}$ be the symbolic filtration of $I$. Then the analytic spread $\ell( \mathcal I_Q)<\mbox{ht}(Q)$ for all prime ideals $Q$ of $R$ which contain $I$ and are not one of the minimal primes $P_i$ of $I$ and $\ell(\mathcal I_{P_i})=\mbox{ht}(P_i)=\mbox{ht}(I)$
    for all minimal primes  $P_i$ of $I$.
    \end{Proposition}

 \begin{Proposition} Let $R$ be a local domain of positive dimension. Let $\mfp$ be a prime ideal in $R$ such that $\mbox{ht}(\mfp)=\dim R-1$ (so that $\dim R/\mfp=1$). Let $d\in \ZZ_{>0}$. If the $\mfp^{(d)}$-adic filtration 
 $\{(\mfp^{(d)})^n\}_{n\in\NN}$ is a 1-divisorial filtration then $\mfp^{(d)}$ is equimultiple. 	
 \end{Proposition} 
  
  \begin{proof} We have that 
  $$
  \dim R-1=\mbox{ht}(\mfp^{(d)})\le \ell(\mfp^{(d)})\le \dim R
  $$
  and $\ell(\mfp^{(d)})\ne \dim R$ by \cite{Mc} or \cite[Theorem 5.4.6]{HS} (or by Theorem \ref{PropAS7} above).
  \end{proof}

  \begin{Proposition}\label{PropAS10} Let $R$ be  a normal, excellent  local domain of dimension three with an isolated singularity and $I$ be  an intersection of (a finite number of) height two prime ideals of $R$. Let $\mathcal I=\{I^{(m)}\}$ be the filtration of symbolic powers of $I$, so that $\mathcal I$ is a 1-divisorial filtration of $R$. 
  Then $\mathcal I$ is not Noetherian if and only if the $a$-th truncation $\mathcal I_a$ of $\mathcal I$ (Definition \ref{trunc}) satisfies 
  $$
  \ell(\mathcal I_a)=3\mbox{ for all $a\in \ZZ_{>0}$}.
  $$
  \end{Proposition}


  \begin{proof}
  We have that  $\ell(\mathcal I)\le2$ by Corollary \ref{CorAS4}. 
 If $\mathcal I$ is a Noetherian filtration, then  $R[\mathcal I_a]=R[\mathcal I]$ for all $a$ sufficiently large, so that, by Proposition  \ref{PropAS11},
$2=\mbox{ht}(I)\le \ell(\mathcal I_a)=\ell(\mathcal I)\le 2$.

 Suppose that  $\mathcal I$ is not  Noetherian. We will show that   $\ell(\mathcal I_a)=3$ for all $a>0$.
 We will prove this statement, by assuming that $\ell(\mathcal I_a)=2$ for some $a$, and deriving a contradiction.
 Write $I=\mfp_1\cap\cdots\cap \mfp_r$ where $\mfp_1,\ldots,\mfp_r$ are height two prime ideals in $R$. Let $\nu_i$ be the ${\mfp_i}_{\mfp_i}$-adic valuation of $R_{\mfp_i}$. Then $I(\nu_i)_n=\mfp_i^{(n)}$ for $1\le i\le r$ and all $n\in \NN$, and $I^{(n)}=I(\nu_1)_n\cap\cdots\cap I(\nu_r)_n$ for all $n\in \NN$.

  Let $\overline {\mathcal I_a}=\{\overline{I_{a,n}}\}$, the filtration of integral closures of the ideals in $\mathcal I_a$. Then $R[\overline{\mathcal  I_a}]$ is finite over $R[\mathcal I_a]$.
 There exists $d>0$ such that $\overline{I_{a,nd}}=(\overline{I_{a,d}})^n$ for all $n\ge 0$,
 $\ell(\mathcal I_a)=\ell(I_{a,d})$ and $\ell(\overline{\mathcal I_a})=\ell(\overline {I_{a,d}})$
 by Remark \ref{RemAS1} and Proposition \ref{PropAS11}. Thus $\ell(\overline{I_{a,d}})=2$.

 Let $\pi:X=\mbox{Proj}(R[\overline{I_{a,d}}])\rightarrow \mbox{Spec}(R)$ be the 
 blow up of $\overline{I_{a,d}}$. $X$ is normal since the ring $R[\overline{I_{a,d}}]=\sum_{n\ge 0}\overline{I_{a,dn}}t^n$ is integrally closed.   Since $\ell(\overline{I_{a,d}})=2$, $\dim \pi^{-1}(m_R)=1$ and so there are no prime divisors on $X$ which contract to $m_R$. Thus
 $\overline{I_{a,d}}\mathcal O_X=\mathcal O_X(-dE)$ where $E=E_1+\cdots+E_r$ is the sum of prime divisors $E_i$ on $X$ such that the valuation $\nu_{E_i}=\nu_i$. 
 Since $X$ and $R$ are normal, we have that 
 $$
 \pi_*\mathcal O_X(-nE)=I(\nu_1)_n\cap\cdots\cap I(\nu_r)_n=\mfp_1^{(n)}\cap\cdots\cap \mfp_r^{(n)}=I^{(n)}
$$ for all $n\in\NN$. 
  
 There exists a graded  exact sequence  
 $$
 0\rightarrow K\rightarrow R[x_0,\ldots,x_m]\rightarrow R[\overline{\mathcal I_{a,d}}]\rightarrow 0,
 $$
 which gives a closed embedding of $X$ into $\PP^m_R$, such that $\mathcal O_{\PP^m_R}(1)\otimes \mathcal O_X\cong \mathcal O_X(-dE)$.
 
 Sheafify this sequence to get  short exact sequences
 $$
 0\rightarrow \mathcal K(n)\rightarrow \mathcal O_{\PP^m_R}(n)\rightarrow \mathcal O_X(-ndE)\rightarrow 0
 $$
 and take global sections to get an exact sequence of $R$-algebras (by \cite[Proposition II.5.13]{Ha})
 $$
R[x_0,\ldots,x_m]=\oplus_{n\ge 0}H^0(\PP^m_R,\mathcal O_{\PP^m_R}(n))\rightarrow \sum_{n\ge 0}I^{(nd)}
\rightarrow \oplus_{n\ge 0}H^1(\PP^m_R,\mathcal K(n)).
$$ 
 We have that $\oplus_{n\ge 0}H^1(\PP^m_R,\mathcal K(n))$ is a finitely generated $R$-module by \cite[Theorem III.5.2(b)]{Ha}.
 Thus $A:=\oplus_{n\ge 0}I^{(nd)}$ is a finitely generated $R$-algebra. 
 
 Since $\mathcal O_X(-dE)$ is an invertible sheaf, for $i,n\in \ZZ$, the reflexive rank 1 sheaf of the Weil divisor $-(i+nd)E$ is
 $\mathcal O_X(-(i+nd)E)\cong \mathcal O_X(-iE)\otimes\mathcal O_X(-ndE)$.
 By \cite[Corollary II.5.18]{Ha}, for $i>0$, there is a short exact sequence of coherent $\mathcal O_X$-modules
 \begin{equation}\label{eqAS2}
 0\rightarrow L\rightarrow \sum_{j=1}^s\mathcal O_X(-n_jdE)\rightarrow \mathcal O_X(-iE)\rightarrow 0.
 \end{equation}
 with $n_j\in \ZZ$. After possibly replacing $i$ with a smaller integer which is equivalent to $i$ modulo $d$, we may assume that all $n_j$ are positive. Now for all $j$, $J_j:=\oplus_{n\geq 0}~ \pi_*(\mathcal O_X(-(n_j+n)dE))$ is a graded ideal in $A$, so it is a finitely generated $A$-module.  From (\ref{eqAS2}) we obtain a short exact sequence of $A$-modules
 $$
 \sum_{j=1}^sJ_j\rightarrow \sum_{n\ge 0}I^{(i+nd)}\rightarrow M.
 $$
 where $M=\sum_{n\ge 0}H^1(X,L\otimes\mathcal O_X(-ndE))$ is a finitely generated $R$-module (again by \cite[Theorem III.5.2(b)]{Ha}). Thus  $\sum_{n\ge 0}I^{(i+nd)}$ is a finitely generated $A$-module, and so $\oplus_{n\ge 0}I^{(n)}$ is a finitely generated $R$-algebra, in contradiction to our assumption. 
 \end{proof}
 
 We have the following theorem, that  uses  examples which will be constructed in  Section \ref{SecEx}.
    
    \begin{Theorem}\label{PropN51} Suppose that $R$ is a regular local ring of dimension 3, and $\mathfrak p$ is a height two prime ideal of $R$. Let $\mathcal I=\{\mathfrak p^{(n)}\}$ be the symbolic filtration of $\mathfrak p$. Then $\ell(\mathcal I)\le 2$ and all values $\ell(\mathcal I)=0,1,2$ can occur.
 \end{Theorem}
  
  \begin{proof} 
  The bound $\ell(\mathcal I)\le 2$ follows from Corollary \ref{CorAS4}. Examples \ref{EXN} and \ref{EXN1} have analytic spread 0 and 1 respectively. A prime ideal $\mathfrak p=(x,y)$ where $x,y$ are part of a regular system of parameters in $R$   gives an example with analytic spread 2.
  \end{proof}

 We have (by Lemma \ref{sp0}) the following ideal theoretic interpretation of analytic spread zero for a symbolic filtration $\mathcal I=\{\mathfrak p^{(n)}\}$. We have that 
    $$
    \ell(\mathcal I)=0\mbox{ if and only if }
          \mbox{ for all $n$ and     $f\in \mathfrak p^{(n)}$, there exists $m>0$ such that $f^m\in m_R\mathfrak p^{(mn)}$.}
  $$  
  
 In Theorem \ref{PropN51}, we necessarily have that the symbolic algebra is not finitely generated if $\ell(\mathcal I)<2$  (by Proposition \ref{PropAS11}). A simple example of a symbolic algebra achieving the maximum analytic spread   $\ell(\mathcal I)=2$   may be constructed by taking $\mathfrak p$ to be a regular prime ideal in $R$ ($\mathfrak p=(x,y)$ where $x,y,z$ is a regular system of parameters in $R$). We do not know of an example such that $\ell(\mathcal I)=2$ but the symbolic algebra is not finitely generated.
 
We  look  a little more closely at the most dramatic case of the theorem, when   $\ell(\mathcal I)=0$.
  By Proposition \ref{PropAS10}, we have that $\ell(\mathcal I_b)=3$ for all truncations $\mathcal I_b$ of $\mathcal I$.
  The analytic spread $\ell(\mathcal I)$ being zero has the following interpretation in the geometry of the canonical projection $\phi:\mbox{Proj}(R[\mathcal I])\rightarrow \mbox{Spec}(R)$.
  We have that  
 $\phi^{-1}(\mathfrak p)=\PP^1_{\kappa(\mathfrak p)}$, where $\kappa(\mathfrak p)=(R/\mathfrak p)_{\mathfrak p}$, since $\mbox{Proj}(\oplus_{n\ge 0}\mathfrak p_{\mathfrak p}^n)$ is the blow up of the maximal ideal $\mathfrak p_{\mathfrak p}$
 in the two dimensional regular local ring $R_{\mathfrak p}$, so that $\dim \phi^{-1}(\mathfrak p)=1$, but
 $\phi^{-1}(m_R)=\emptyset$
  since $\ell(\mathcal I)=0$. In particular, the theorem on upper semicontinuity of fiber dimension (\ref{eqN2}) for ideals  fails in this non Noetherian situation.  
  
  Theorem \ref{PropN51} shows that the inequality (\ref{eqN3}) for ideals fails for symbolic filtrations, as we see by taking $\mathfrak p$ in Theorem \ref{PropN51} such that $\ell(\mathcal I)<2$, so that $2=\mbox{ht}(\mathcal I)>\ell(\mathcal I)$.

 \section{Some Examples of Symbolic Algebras}\label{SecEx}  In this section we use famous examples by Nagata and Zariski to 
 compute the analytic spread of some space curve singularities and some related examples.


  \begin{Example}\label{EXN} Suppose that $a\ge 0$. Then there exists a prime ideal $Q$ of height $2+a$  in a regular local ring $A$ of dimension $3+a$ such that $\ell(\mathcal J)=a$, where $\mathcal J=\{Q^{(n)}\}$ is the 1-divisorial  filtration on $A$ of symbolic powers of $Q$.
  \end{Example} 
  

  We make use of a famous example of Nagata.
   Let $s$ be a positive integer with $s\ge 4$, and $r=s^2$. Let $\alpha_1,\ldots,\alpha_r\in \PP^2_{\CC}$ be independent generic points of $\PP^2$ over $\QQ$. 
  
Let $\mathcal I_{\alpha_i}$ be the ideal sheaf of $\alpha_i$ in $\PP^2$ and let $H'$ be a linear hyperplane section of $\PP^2$.

  The difficult statement of Theorem \ref{ThmN1} is proven by Nagata in \cite{N2} and in Proposition 1 of Chapter 3, page 18 of \cite{N1}.
  
  \begin{Theorem}\label{ThmN1}(Nagata) Let notation be as above. 
  \begin{enumerate}\item[1)] Suppose that $d,m_1,\ldots,m_r\in \NN$ and $H^0(\PP^2,\mathcal O_{\PP^2}(dH')\otimes \mathcal I_{\alpha_1}^{m_1}\otimes\cdots\otimes \mathcal I_{\alpha_r}^{m_r})\ne 0$. Then 
  $$
  d>\frac{1}{\sqrt{r}}\sum_{i=1}^rm_i.
  $$
  \item[2)] Suppose that $r'$ is a real number such that $r'>\sqrt{r}$. Then there exist $d,m\in \ZZ_{>0}$ such that $r'>\frac{d}{m}>\sqrt{r}$ and $H^0(\PP^2,\mathcal O_{\PP^2}(dH')\otimes \mathcal I_{\alpha_1}^{m}\otimes\cdots\otimes \mathcal I_{\alpha_r}^{m})\ne 0$.
    \end{enumerate}
  \end{Theorem}
  
  Let $\Lambda:X\rightarrow \PP^2$ be the blow up of the points $\alpha_1,\ldots,\alpha_r$ with exceptional lines $E_1,\ldots,E_r$.   
  Let $H=\Lambda^*(H')$. Since $\Lambda$ is the blowup of the points $\alpha_1,\ldots,\alpha_r$ on the nonsingular surface $\PP^2$, we have that for all $d,m_1,\ldots,m_r\ge 0$,
  $$
  H^0(X,\mathcal O_X(dH-m_1E_1-\cdots-m_rE_r))=H^0(\PP^2,\mathcal O_{\PP^2}(dH')\otimes \mathcal I_{\alpha_1}^{m_1}\otimes\cdots\otimes \mathcal I_{\alpha_r}^{m_r}).
  $$  
  Let $E=E_1+\cdots+E_r$. The canonical divisor $K_X$ on $X$ is $K_X=-3H+E$.

\begin{Lemma}\label{LemmaN0} Let notation be as above.
\begin{enumerate}
\item[1)] Let $C$ be an irreducible reduced curve on $X$ with $C\ne E_i$ for any $i$.  Then $C\sim dH-\sum m_iE_i$ for some $d,m_i\in \NN$ with $d>0$.
\item[2)] Let $dH-mE$ be a divisor with $d\ge -2$. Then $H^2(X,\mathcal O_X(dH-mE))=0$.
\item[3)] Let $L=dH-mE$ with $d,m\in \ZZ_{>0}$. Then $L$ is ample if $\frac{d}{m}>\sqrt{r}$.
\item[4)] Let $L=dH-mE$ with $d,m\in \ZZ_{>0}$. Then   $H^1(X,\mathcal O_X(dH-mE))=0$ if  $d>\sqrt{r}m+(\sqrt{r}-3)$.
\item[5)] Suppose that $d,m\ge 0$. Then
  $$
  H^0(X,\mathcal O_X(dH-mE))=H^0(X,\mathcal O_X(H))H^0(X,\mathcal O_X((d-1)H-mE))
  $$ 
  if $d\ge \sqrt{r}m+\sqrt{r}$.

\end{enumerate}
  \end{Lemma}
  
  \begin{proof} 1). $C\sim dH-\sum m_iE_i$ for some $d,m_i\in \ZZ$. We have $(H\cdot C)=d\ge 0$ since the complete linear system $|H|$ is base point free.  Further, $(E_i\cdot C)=m_i\ge 0$ for $1\le i\le r$. There exists $e>0$ such that $eH-E$ is ample. 
  If $d=0$, so that $C\sim \sum -m_iE_i$ with some $m_i>0$, then $((eH-E)\cdot C)=-\sum m_i<0$, which is impossible. 
  
  2). By Serre duality
  $$
  H^2(X,\mathcal O_X(dH-mE))\cong H^0(X,\mathcal O_X(mE-dH+K_X)=H^0(X,\mathcal O_X(-(d+3)H+(m+1)E)).
  $$
   The complete linear system $|H|$ is base point free on $X$ and 
   $$
   (H\cdot(mE-dH+K_X))=-(d+3)<0\mbox{ for }d\ge -2,
   $$
    so
    $H^0(X,\mathcal O_X(mE-dH+K_X)=0$.    
    
    3). Suppose that $C$ is an irreducible reduced curve on $X$. 
  Then  by 1) and Theorem \ref{ThmN1}, $C$ is linear equivalent to $eH-\sum n_iE_i$ with $e,n_1,\ldots,n_r\in \NN$ and  $e>\frac{1}{\sqrt{r}}\sum_{i=1}^rn_i$. Hence $(C\cdot L)=de-m\sum_{i=1}^rn_i>0$. Further $(L^2)=d^2-m^2r>0$, so $L$ is ample by the Nakai Moishezon criterion (\cite[Theorem V.1.10]{Ha}).
  
  4). The  divisor 
  $(dH-mE)-K_X$ is ample if $d>\sqrt{r}m+(\sqrt{r}-3)$ by 3).
  Thus, by the Kodaira vanishing theorem (\cite[Remark III.7.15]{Ha}), 
  $$
  H^1(X,\mathcal O_X(dH-mE))=H^1(X,\mathcal O_X((dH-mE-K_X)+K_X))=0
  $$ 
   if  $d>\sqrt{r}m+(\sqrt{r}-3)$.
  
  5). The statements 2) and 4) imply that $\mathcal O_X(-mE)$ is $d$-regular if $d>\sqrt{r}m+(\sqrt{r}-2)$; that is, $H^i(X,\mathcal O_X(-mE)\otimes \mathcal O_X((d-i)H))=0$ for $i=1,2$. Thus the conclusions of 5) hold by page 99 \cite{M} (also proven in \cite[Theorem 17.35]{AG}).   
    \end{proof}

  Let $S=\CC[x_1,x_2,x_3]$ be the homogeneous coordinate ring of $\PP^2$ and $\mathfrak m$ be the graded maximal ideal of $S$. Let $P_i$ be the height two prime ideal in $S$ of the point $\alpha_i$ for $1\leq i\leq r$.   
  Then   $$
  S=\oplus_{d\ge 0}H^0(\PP^2,\mathcal O_{\PP^2}(dH'))=\oplus_{d\ge 0}H^0(X,\mathcal O_X(dH))
  $$
  and
   $$
   \mathfrak m=\oplus_{d>0}H^0(\PP^2,\mathcal O_{\PP^2}(dH'))=\oplus_{d> 0}H^0(X,\mathcal O_X(dH)).
   $$
 For $d,n_1,\ldots,n_r\in \NN$,
  $$
  \begin{array}{l}
  H^0(\PP^2,\mathcal O_{\PP^2}(dH')\otimes\mathcal I_{\alpha_1}^{n_1}\otimes\cdots\otimes \mathcal I_{\alpha_r}^{n_r})\\
  =\{F\in S\mid \mbox{$F$ is homogeneous of degree $d$ and $F\in P_1^{n_1}\cap\cdots\cap P_r^{n_r}$}\}
  \end{array}
  $$
  and
 $$
  \begin{array}{lll}
  P_1^{n_1}\cap\cdots \cap P_r^{n_r}&=&\bigoplus_{d>0}H^0(\PP^2,\mathcal O_{\PP^2}(dH')\otimes \mathcal I_{\alpha_1}^{n_1}\otimes\cdots\otimes\mathcal I_{\alpha_r}^{n_r})\\
  &=&\bigoplus_{d>0}H^0(X,\mathcal O_X(dH-n_1E_1-\cdots-n_rE_r).  
  \end{array}
  $$
  In particular,
  $$
  P_1^{n}\cap\cdots \cap P_r^{n}=  \bigoplus_{d>0}H^0(X,\mathcal O_X(dH-nE)).
  $$  
  Recall that $s=\sqrt{r}\in \ZZ_+$ (with $s\ge 4$).
  
  \begin{Proposition}\label{PropN5} Given $n\in \ZZ_{>0}$,
     $$
  \left(P_1^{n}\cap\cdots\cap P_r^{n}\right)^{s}\subset \mathfrak m\left(P_1^{sn}\cap\cdots\cap P_r^{sn}\right).
  $$
  \end{Proposition}
  
  \begin{proof} Since $S$ is graded and Noetherian,
    it suffices to show that if $0\ne f\in H^0(X,\mathcal O_X(dH-nE))$, then $f^s\in H^0(X,\mathcal O_X(H))H^0(X,\mathcal O_X(sd-1)H-snE)$. We have that $d>sn$ by Theorem \ref{ThmN1}. The statement then 
  follows from 5) of Lemma \ref{LemmaN0}, since 
  $f^s\in H^0(X,\mathcal O_X(sdH-snE))$ and 
  $sd\ge s(ns+1)= s(sn)+s$.
  \end{proof} 
  
  Let $R=S_{\mathfrak m}$, a three dimensional regular local ring, with maximal ideal $m_R=\mathfrak mS_{\mathfrak m}$.
  Let $\mathfrak p_i=(P_i)_{\mathfrak m}$ for $1\le i\le r$. The ideals $\mathfrak p_i$ are height two prime ideals in $R$.
  Let $\mathcal I=\{I_n\}$ where $I_n=\mathfrak p_1^{n}\cap \cdots\cap\mathfrak p_r^{n}$. The filtration $\mathcal I$ is the 1-divisorial  filtration on $R$, consisting of the symbolic powers  $I_n=I^{(n)}$ of $I=I_1$. Thus $\mbox{ht}(\mathcal I)=\mbox{ht}(I)=2$.

  By Proposition \ref{PropN5}, we have that for all $n>0$,
  \begin{equation}\label{eqN7}
 (I^{(n)})^s= \left(\mfp_1^{n}\cap\cdots\cap \mfp_r^{n}\right)^{s}\subset  m_R\left(\mfp_1^{sn}\cap\cdots\cap \mfp_r^{sn}\right)=m_RI^{(sn)}.
  \end{equation}  
  
  

  We will first construct the example when $a=0$. In \cite{Ro}, a height two prime ideal 
  $\mathfrak p$ in $R=\CC[x_1,x_2,x_3]_{(x_1,x_2,x_3)}$ and a continuous $\CC$-algebra  isomorphism $\phi:\hat R\rightarrow \hat R$ such that 
  $\phi(\widehat{I^{(n)}})=\widehat{\mfp^{(n)}}$ for all $n\ge 0$ are constructed, where $I$ is the ideal defined before (\ref{eqN7}).
  Let $s=\sqrt{r}\in \ZZ_{>0}$ be the integer defined before Theorem \ref{ThmN1}.
  By (\ref{eqN7}), $(I^{(n)})^s\subset m_RI^{(ns)}$ for all $n> 0$. Thus $(\widehat{I^{(n)}})^s\subset m_{\hat R}\widehat{I^{(ns)}}$, so applying $\phi$, we obtain
  $$
  \widehat{(\mfp^{(n)})^s}=(\widehat{\mfp^{(n)}})^s\subset m_{\hat R}\widehat{\mfp^{(ns)}}=\widehat{m_R\mfp^{(ns)}}.
  $$
  Since $R\rightarrow\hat R$ is faithfully flat, we obtain
   that  for all $n> 0$, we have that
  \begin{equation}\label{eqAS20} 
 \left( \mfp^{(n)}\right)^s\subset m_R\mfp^{(sn)}.
 \end{equation}
 Let $\nu$ be the $\mfp_{\mfp}$-adic valuation of $R_{\mfp}$, which is the discrete valuation of $K:=\CC(x_1,x_2,x_3)$ such that 
 the valuation ideals of $\nu$ in $R$ are $I(\nu)_n=\mfp^{(n)}$ for all $n\ge 0$.
 
 Let $A=\CC[x_1,x_2,x_3,y_1,\ldots,y_a]_{(x_1,x_2,x_3,y_1,\ldots,y_a)}$, the polynomial ring over $\CC$ in the variables $x_1,x_2,x_3,y_1,\ldots,y_a$,
 and $Q=\mfp A+(y_1,\ldots,y_a)$,  a prime ideal of height $2+a$ in $A$. Let $\omega$ be the Gauss valuation of $L:=\CC(x_1,x_2,x_3,y_1,\ldots,y_a)$, defined by 
 $$
 \omega(f)=\min\{\nu(b_{i_1,\ldots,i_a})+i_1+\cdots+i_a\}
 $$
 if $f=\sum b_{i_1,\ldots,i_a}y_1^{i_1}\cdots y_a^{i_a}\in K[y_1,\ldots,y_a]$ with $b_{i_1,\ldots,i_a}\in K$ for all $i_1,\ldots,i_a$. The valuation $\omega$ is a discrete valuation of $L$ which dominates $A_Q$. Since 
 $$
 A_Q=\left(R_{\mfp}[y_1,\ldots,y_a]\right)_{QR_{\mfp}[y_1,\ldots,y_a]},
 $$  
  we have that $\omega$ is the $Q$-adic valuation of $A_{Q}$. Thus the valuation ideals of $\omega$ in $A$ are $I(\omega)_n=Q^{(n)}$ for $n\ge 0$. Let $\mathcal J=\{Q^{(n)}\}$. The filtration $\mathcal J$ is a 1-divisorial filtration on $A$.
  
  \begin{Proposition}\label{PropAS22} Let $N=\sqrt{m_AA[\mathcal J]}$ be the radical of $m_AA[\mathcal J]$ in $A[\mathcal J]$. Let $\overline y_1,\ldots \overline y_a$ be the respective classes of $y_1t,\ldots,y_at$ in $A[\mathcal J]/N$. Then 
  $$
  A[\mathcal J]/N=\CC[\overline y_1,\ldots,\overline y_a]
  $$
  is a standard graded polynomial ring in the variables $\overline y_1,\ldots,\overline y_a$ over $\CC$.
  \end{Proposition}
  
  \begin{proof} For $n> 0$,
  $$
  Q^{(n)}=I(\omega)_n=(y_1,\ldots,y_a)^n+(y_1,\ldots,y_a)^{n-1}\mfp+(y_1,\ldots,y_a)^{n-2}\mfp^{(2)}+\cdots+\mfp^{(n)}A.
  $$
  We first show that $A[\mathcal J]/N$ is generated by $\overline y_1,\ldots,\overline y_a$ as a $\CC$-algebra. Since $N$ is graded, it suffices to show that if $f\in \mfp^{(d)}$ with $d>0$ and $y_1^{i_1}\cdots y_a^{i_a}$ is such that $i_1,\ldots,i_a\in \NN$ and $i_1+\cdots+i_a=n-d$, then 
  $(y_1^{i_1}\cdots y_a^{i_a})^sf^s\in m_AQ^{(sn)}$. Since $(i_1+\cdots+i_a)s=(n-d)s$, it suffices to show that $f^s\in m_R\mfp^{(sd)}$, which follows from (\ref{eqAS20}).
  
 We now show that the standard graded $\CC$-algebra $\CC[\overline y_1,\ldots,\overline y_a]$ is a polynomial ring over $\CC$. Suppose otherwise. We will find a contradiction. Then  for some $n>0$, there is a relation
 $$
 \sum_{i_1+\cdots+i_a=n}\lambda_{i_1,\ldots,i_a}\overline y_1^{i_1}\cdots \overline y_a^{i_1}=0
 $$
 for some $\lambda_{i_1,\ldots,i_a}\in \CC$ not all zero. Let $G=\sum_{i_1+\cdots+i_a=n}\lambda_{i_1,\ldots,i_a}y_1^{i_1}\cdots y_a^{i_a}\in Q^{(n)}$. Since $G\in N$, there exists $m>0$ such that $G^m\in m_AQ^{(mn)}$. Now $G^m$ has $m_A$-order $mn$, and every element of $m_AQ^{(mn)}$ has $m_A$-order $\ge mn+1$. Thus $G=0$, a contradiction to the assumption that some $\lambda_{i_1,\ldots,i_a}\ne 0$.
  \end{proof}
  Example \ref{EXN} thus has analytic spread $\ell(\mathcal J)=a$.


  \begin{Example}\label{EXN1} There exists a prime ideal $\mathfrak p$ of height $2$  in a regular local ring $R$ of dimension $3$ such that $\ell(\mathcal J)=1$, where $\mathcal J=\{\mathfrak p^{(n)}\}$ is the 1-divisorial  filtration on $R$ of symbolic powers of $\mathfrak p$.
  \end{Example} 
  
  
  

  We make use of a famous example of Zariski \cite{Z1},  expositions of which can be found in \cite[Section 2.3]{La} and \cite[Theorem 20.14]{AG}.
   Let $\alpha_1,\ldots,\alpha_{12}\in \PP^2_{\CC}$ be independent generic points of an elliptic curve $E'$ of $\PP^2$ over $\QQ$. 
   The curve $E'$ is defined by the vanishing of an irreducible  cubic form $G\in \CC[x_1,x_2,x_3]$. 
  
Let $\mathcal I_{\alpha_i}$ be the ideal sheaf of $\alpha_i$ in $\PP^2$ and let $H'$ be a linear hyperplane section of $\PP^2$.

  Let $\Lambda:X\rightarrow \PP^2$ be the blow up of the points $\alpha_1,\ldots,\alpha_{12}$ with exceptional lines $F_1,\ldots,F_{12}$.   
  Let $H=\Lambda^*(H')$. Since $\Lambda$ is the blowup of the points $\alpha_1,\ldots,\alpha_{12}$ on the nonsingular surface $\PP^2$, we have that for all $d,m_1,\ldots,m_r\ge 0$,
  $$
  H^0(X,\mathcal O_X(dH-m_1F_1-\cdots-m_{12}F_{12}))=H^0(\PP^2,\mathcal O_{\PP^2}(dH')\otimes \mathcal I_{\alpha_1}^{m_1}\otimes\cdots\otimes \mathcal I_{\alpha_r}^{m_r}).
  $$  
  Let $F=F_1+\cdots+F_{12}$. The canonical divisor $K_X$ on $X$ is $K_X=-3H+F$. Let $E$ be the strict transform of  $E'$ on $X$. We have that $\Lambda^*(E')=E+F$, where $E$ is an elliptic curve on $X$ which is isomorphic to $E'$, $(E\cdot E)=-3$ and 
  $\mathcal O_X(H+E)\otimes \mathcal O_E$ is a degree 0 invertible sheaf on $X$ of infinite order, so that
  \begin{equation}\label{eqN12}
  H^0(E,\mathcal O_X(m(H+E))\otimes \mathcal O_E)=0\mbox{ for all nonzero integers $m$.}
  \end{equation}
    Further, $(F\cdot F)=-12$.

\begin{Lemma}\label{LemmaN10} Let notation be as above.
\begin{enumerate}
\item[1)] Let $C$ be an irreducible reduced curve on $X$ with $C\ne F_i$ for any $i$.  Then $C\sim dH-\sum m_iF_i$ for some $d,m_i\in \NN$ with $d>0$.
\item[2)] Let $dH-mF$ be a divisor with $d\ge -2$. Then $H^2(X,\mathcal O_X(dH-mF))=0$.
\item[3)] Let $L=dH-mF$ with $m\in \ZZ_{>0}$ and $d>4m$. Then $L$ is ample.
\item[4)] Let $\sigma\in H^0(X,\mathcal O_X(3H-F))$ be the section whose divisor is $E$. Then 
$$
\mathcal O_X(3H-F)=\sigma\mathcal O_X
$$
 and $H^0(X,\mathcal O_X(3H-F))=\sigma\CC$. Further, 
 $$
 H^0(X,\mathcal O_X(3mH-mF))=H^0(X,\mathcal O_X(3H-F))^m
 $$
  for all $m\in \ZZ_{>0}$.
  \item[5)] Suppose that $d\ge 0$, $m\ge 0$ and $d<3m$. then $H^0(X,\mathcal O_X(dH-mF))=0$.
  \item[6)] Let $L=dH-mF$ with $d,m\in \ZZ_{>0}$. Then   $H^1(X,\mathcal O_X(dH-mF))=0$ if  $d>4m+1$.
  \item[7)] Suppose that $d,m\ge 0$. Then
  $$
  H^0(X,\mathcal O_X(dH-mF))=H^0(X,\mathcal O_X(H))H^0(X,\mathcal O_X((d-1)H-mF))
  $$
  if $d\ge 4m+4$.

\end{enumerate}
  \end{Lemma}
  
  \begin{proof} The proofs of 1) and 2) are the same as the proofs of 1) and 2) of Lemma \ref{LemmaN0}.
  
  3). Suppose $C$ is an integral curve on $X$ other than $F_i$ or $E$. Then $(C\cdot E)\ge 0$ and $(C\cdot H)=(C\cdot\Lambda^*(H'))=(\Lambda_*(C)\cdot H')>0$ so $(C\cdot L)>0$, since $L\sim mE+(d-3m)H$.  Further, $(L\cdot E)=3d-12m>0$ and 
  $(L\cdot F_i)=m>0$ for $1\le i\le 12$, so every irreducible curve of $X$ has positive intersection number with $L$.
  Finally, $(L\cdot L)=d^2-12m^2>4m^2>0$, so $L$ is ample by the Nakai Moishezon criterion.
  
 4). We have that $mE\sim 3mH-mF$. Suppose that $D$ is an effective divisor on $X$ which is linearly equivalent to $3mH-mF$. Then $(E\cdot (3mH-mF))=(E\cdot mE)=-3m<0$ so $E$ is in the support of $D$, so $D-E$ is effective. By induction on $m$, we obtain $D=mE$.  
  
  5). Suppose there exists an effective divisor $D$ such that $D\sim dH-mF$. We compute
  $$
  (D\cdot E)=((dH-mF)\cdot(3H-F))=3d-12m<-3m.
  $$
  Thus $E$ is in the support of $D$, so $D-E\sim (d-3)H-(m-1)F$ is effective, with $d-3<3(m-1)$. Continuing in this way, we obtain that $(d-3m)H-2mF$ is an effective divisor, which is a contradiction to 1).
  
  The proof of 6) is as the proof of 4) of Lemma \ref{LemmaN0}, using 3) of this lemma.
  
  The proof of 7) is as the proof of 5) of Lemma \ref{LemmaN0}, using 6) and 2) of this lemma.
  
    \end{proof}  
    
    \begin{Lemma}\label{LemmaN11} Suppose that $0<3m<d\le 4m$. Then
    \begin{enumerate}
    \item[1)] $H^0(X,\mathcal O_X(dH-mF))>0$.
    \item[2)] $H^0(X,\mathcal O_X(dH-mF))=H^0(X,\mathcal O_X(3H-F))H^0(X,\mathcal O_X((d-3)H-(m-1)F)$ where $d-3>3(m-1)$.
    \item[3)] $H^0(X,\mathcal O_X(dH-mF))=H^0(X,\mathcal O_X(3H-F))^rH^0(X,\mathcal O_X(d'H-m'F))$ where
    $r=4m-d +1$, $d'=d-3r$, $m'=m-r$ satisfy $d'>4m'$.
    \end{enumerate}
    \end{Lemma}
    
    \begin{proof} 1). We have that $dH-mF=m(3H-F)+(d-3m)H\sim mE+(d-3m)H$, which is a nonzero effective divisor. 
    
    2). Recall that $\sigma$ is a section of $\mathcal O_X(3H-F)$ whose divisor is $E$. Tensor
    $$
    0\rightarrow \mathcal O_X(-(3H-F))\stackrel{\sigma}{\rightarrow} \mathcal O_X\rightarrow \mathcal O_E\rightarrow 0
    $$
    with $\mathcal O_X(dH-mF)$ to get the exact sequence
    $$
    0\rightarrow H^0(X,\mathcal O_X((d-3)H-(m-1)F))\stackrel{\sigma}{\rightarrow} H^0(X,\mathcal O_X(dH-mF))\rightarrow H^0(E,\mathcal O_X(dH-mF)\otimes \mathcal O_E).
    $$
    The rightmost vector space is zero since $((dH-mF)\cdot E)=3d-12m\le0$, and by (\ref{eqN12}).
    
    3). For $t\in \QQ$, we have that 
    $(d-3t)H-(m-t)F$ satisfies $d-3t\le 4(m-t)$ if and only if $t\le 4m-d$.
     Now 3) follows from 1) and 2) of this lemma.
    
    \end{proof}

  Let $S=\CC[x_1,x_2,x_3]$ be the homogeneous coordinate ring of $\PP^2$ and $\mathfrak m$ be the graded maximal ideal of $S$. Let $P_i$ be the height two prime ideal in $S$ of the point $\alpha_i$ for $1\leq i\leq 12$. Then 
  $$
  S=\oplus_{d\ge 0}H^0(\PP^2,\mathcal O_{\PP^2}(dH'))=\oplus_{d\ge 0}H^0(X,\mathcal O_X(dH)),
  $$
   $$
   \mathfrak m=\oplus_{d>0}H^0(\PP^2,\mathcal O_{\PP^2}(dH'))=\oplus_{d> 0}H^0(X,\mathcal O_X(dH))
   $$
 and for $n_1,\ldots,n_{12}\ge 0$,
  $$
  \begin{array}{lll}
  P_1^{n_1}\cap\cdots \cap P_{12}^{n_{12}}&=&\bigoplus_{d>0}H^0(\PP^2,\mathcal O_{\PP^2}(dH')\otimes \mathcal I_{\alpha_1}^{n_1}\otimes\cdots\otimes\mathcal I_{\alpha_{12}}^{n_{12}})\\
  &=&\bigoplus_{d>0}H^0(X,\mathcal O_X(dH-n_1F_1-\cdots-n_{12}F_{12}).  
  \end{array}
  $$
  In particular,
  $$
  P_1^{n}\cap\cdots \cap P_{12}^{n}=  \bigoplus_{d>0}H^0(X,\mathcal O_X(dH-nF)).
  $$  
  By  5) of Lemma \ref{LemmaN10}, 
  \begin{equation}\label{eqN52}
  P_1^{n}\cap\cdots \cap P_{12}^{n}=  \bigoplus_{d\ge 3n}H^0(X,\mathcal O_X(dH-nF)).
  \end{equation}  
  The irreducible cubic form $G$ defining $E'$ is in $H^0(X,\mathcal O_X(3H-F))$ and 
  $$
  h^0(X,\mathcal O_X(3nH-nF))=1\mbox{ for $n>0$}
  $$
   by 4) of Lemma \ref{LemmaN10}, so 
  \begin{equation}\label{N51}
  H^0(X,\mathcal O_X(n(3H-F))=G^n\CC\mbox{ for }n>0.
  \end{equation}
    
  \begin{Proposition}\label{PropN50} Given $n\in \ZZ_{>0}$, and 
     $h \in H^0(X,\mathcal O_X(dH-nF))$ with  $d>3n$, there exists $s\in \ZZ_{>0}$ such that 
     $$
     h^s\in 
      \mathfrak m\left(P_1^{sn}\cap\cdots\cap P_{12}^{sn}\right).
      $$
  \end{Proposition}
  
  \begin{proof}  First suppose that $d>4n$. Then $4d\ge 4(4n)+4$ implies
  $$
  H^0(X,\mathcal O_X(4dH-4nF))=H^0(X,\mathcal O_X(H))H^0(X,\mathcal O_X(4d-1)H-4nF))
  $$
  by 7) of Lemma \ref{LemmaN10}. Thus $h^4\in \mathfrak m\left(P_1^{4n}\cap\cdots\cap P_{12}^{4n}\right)$.  
  
  Now suppose that $3n<d\le 4n$. Then by 3) of Lemma \ref{LemmaN11}, 
  $$
  H^0(X,\mathcal O_X(dH-nF))=H^0(3H-F))^rH^0(X,\mathcal O_X(d'H-m'F))
  $$
  for suitable $r,d',m'$ where $d'>4m'$ and $n=m'+r$. By the first part of this proof,
  $$
  \begin{array}{lll}
  h^4&\in& H^0(X,\mathcal O_X(H))H^0(X,\mathcal O_X(3H-F))^{4r}H^0(X,\mathcal O_X((4d'-1)H-4m'F))\\
  &\subset &\mathfrak m \left(P_1^{4n}\cap\cdots\cap P_{12}^{4n}\right).
  \end{array}
  $$  
    \end{proof} 
  
  Let $R=S_{\mathfrak m}$, a three dimensional regular local ring, with maximal ideal $m_R=\mathfrak mS_{\mathfrak m}$.
  Let $\mathfrak p_i=(P_i)_{\mathfrak m}$ for $1\le i\le 12$. The ideals $\mathfrak p_i$ are height two prime ideals in $R$.
  Let $\mathcal I=\{I_n\}$ where $I_n=\mathfrak p_1^{n}\cap \cdots\cap\mathfrak p_{12}^{n}$. The filtration $\mathcal I$ is the 1-divisorial  filtration on $R$ consisting of the symbolic powers  $I_n=I^{(n)}$ of $I=I_1$. Thus $\mbox{ht}(\mathcal I)=\mbox{ht}(I)=2$.

  
  

   \begin{Proposition}\label{PropAS221} Let $N=\sqrt{m_RR[\mathcal I]}$ be the radical of $m_RR[\mathcal I]$ in $R[\mathcal I]$.  Then 
  $$
  R[\mathcal I]/N\cong \CC[Gt]
  $$
  is a standard graded polynomial ring. 
  \end{Proposition}
  
  \begin{proof}  
  $$
  R[\mathcal I]/N \cong \left(\sum_{n\ge 0} \mathfrak p_1^{(n)}\cap \cdots \cap \mathfrak p_{12}^{(n)}t^n\right)/\sqrt{\mathfrak m
  \left(\sum_{n\ge 0} \mathfrak p_1^{(n)}\cap \cdots \cap \mathfrak p_{12}^{(n)}t^n\right) }\cong \CC[Gt].
  $$ 
  by (\ref{eqN52}), (\ref{N51}) and Proposition \ref{PropN50}.
  \end{proof}
  
   By the method of  \cite{Ro}, we construct a height two prime ideal 
  $\mathfrak p$ in $R=\CC[x_1,x_2,x_3]_{(x_1,x_2,x_3)}$ and a continuous $\CC$-algebra  isomorphism $\phi:\hat R\rightarrow \hat R$ such that $\phi(\hat m_R)=\hat m_R$ and
  $\phi(\widehat{I^{(n)}})=\widehat{\mfp^{(n)}}$ for all $n\ge 0$, where $I$ is the ideal defined before Proposition \ref{PropAS221}. We have that $\hat{\mathfrak p}^{(n)}= \widehat{\mathfrak p^{(n)}}$ and
  $\hat I^{(n)}=\widehat{I^{(n)}}$ for all $n$, as explained in the proof of \cite[Proposition 1]{Ro}.
  
  Since the Rees algebras of all truncations of $\{\mfp^{(n)}\}$ are excellent, we have that
  $$
  \hat R\sqrt{m_RR[\mfp t, \mfp^{(2)}t^2,\ldots,\mfp^{(a)}t^a]}=\sqrt{ \hat m_R\hat R[\hat \mfp t, \hat \mfp^{(2)}t^2,\ldots,\hat \mfp^{(a)}t^a]}
  $$
  for all $a\in \ZZ_{>0}$ and so 
    $$
   \hat R\sqrt{m_RR[\{\mfp^{(n)}\}]}=\sqrt{ \hat m_R\hat R[\{\hat \mfp^{(n)}\}]}.
  $$
  We thus have that 
  $$
  \begin{array}{l}
  R[\{\mathfrak p^{(n)}\}]/\sqrt{m_RR[\{\mathfrak p^{(n)}\}]}\cong \left(R[\{\mathfrak p^{(n)}\}]/\sqrt{m_RR[\{\mathfrak p^{(n)}\}]}\right)\otimes_R\hat R\cong
  \hat R[\{\widehat{\mathfrak p^{(n)}}\}]/\sqrt{\hat m_R \hat R[\{\widehat{\mathfrak p^{(n)}}\}]}\\
  \cong  \hat R[\{\widehat{\mathcal I^{(n)}}\}]/\sqrt{\hat m_R \hat R[\{\widehat{\mathcal I^{(n)}}\}]}\cong\left(R[\mathcal I]/\sqrt{m_RR[\mathcal I]}\right)\otimes_R\hat R  \cong R[\mathcal I]/\sqrt{m_RR[\mathcal I]}\cong \CC[t]
  \end{array}
  $$  
  by Proposition \ref{PropAS221}.
  Thus $\mathcal J=\{\mathfrak p^{(n)}\}$ fulfills the conditions of the example.

\end{document}